\documentclass[11pt]{article}
\pdfoutput=1

\usepackage[backend=biber,
            bibstyle=phys,
            bibencoding=utf8,
            texencoding=utf8,
            sorting=none,
            url=false,
            doi=false,
            eprint=true,
            sortcites=true,
            citestyle=numeric-comp]{biblatex}
\DeclareFieldFormat*{title}{\mkbibitalic{#1}}
\DeclareFieldFormat{labelnumberwidth}{\mkbibbrackets{#1}~~}
\renewbibmacro*{volume+number+eid}{%
  \printfield{volume}%
  \iffieldundef{number}{}{%
  \printtext{(}%
  }%
  \printfield{number}%
  \iffieldundef{number}{}{%
  \printtext{)}%
  }%
  \setunit{\addcomma\space}%
  \printfield{eid}}

\usepackage{amsfonts,amsthm,amsmath,amssymb}
\usepackage{latexsym,xcolor}
\usepackage{mathrsfs}
\usepackage{pifont}
\usepackage{tikz}
\usepackage[all]{xy}
\usepackage[breaklinks]{hyperref}
\usepackage{authblk}

\usepackage[utf8]{inputenc}
\usepackage[T1]{fontenc}

\newtheoremstyle{rem}{3pt}{3pt}{}{}% <Space above><Space below><Body font> <Indent amount>
{\bfseries}{.}{.5em}{}% <Theorem head font><Punctuation after theorem head><Space after theorem headi><Theorem head spec>

\newtheorem{theo}{Theorem}[section]
\newtheorem*{theo*}{Theorem}

\newtheorem{prop}[theo]{Proposition}

\newtheorem{coro}[theo]{Corollary}

\theoremstyle{rem}

\theoremstyle{definition}

\newtheorem*{term*}{Notation/Terminology}

\renewcommand{\sl}{sl}

\newcommand{\pFq}[5]{{
{}_{#1}F_{#2}\left( \genfrac{}{}{0pt}{0}{#3}{#4}\Big| #5\right)
}}

\usepackage{geometry}
\title{\bf Matrix elements of $SO(3)$ in $\sl_3$ representations \\
as bispectral multivariate functions}

\renewcommand*{\Affilfont}{\normalsize\small}
\author[1]{Nicolas Cramp\'e\,}
\author[2]{Julien Gaboriaud\,}
\author[3]{Lo\"ic Poulain d'Andecy\,}
\author[4,5]{Luc Vinet\,\vspace{.5em}}
\affil[1]{Institut Denis-Poisson CNRS/UMR 7013 - Universit\'e de Tours - Universit\'e
d'Orl\'eans,
\newline\vspace{.9em}
Parc de Grandmont, 37200 Tours, France.}
\affil[2]{\vspace{.9em}Graduate School of Informatics, Kyoto University, Sakyo-ku, Kyoto,
606-8501, Japan.
}
\affil[3]{Laboratoire de math\'ematiques de Reims UMR 9008,
Universit\'e de Reims Champagne-Ardenne,
\newline\vspace{.9em}
Moulin de la Housse BP 1039, 51100 Reims, France.}
\affil[4]{Centre de Recherches Math\'ematiques, Universit\'e de Montr\'eal, P.O. Box 6128,
\newline\vspace{.9em}
Centre-ville Station, Montr\'eal (Qu\'ebec), H3C 3J7, Canada.}
\affil[5]{IVADO, Montr\'eal (Qu\'ebec), H2S 3H1, Canada.\vspace{1.5em} }
{
 \makeatletter
 \renewcommand\AB@affilsepx{: \protect\Affilfont}
 \makeatother
 \affil[ ]{E-mail addresses}
 \makeatletter
 \renewcommand\AB@affilsepx{, \protect\Affilfont}
 \makeatother
 \affil[1]{crampe1977@gmail.com}
 \affil[2]{julien.gaboriaud.57f@st.kyoto-u.ac.jp}
 \affil[3]{loic.poulain-dandecy@univ-reims.fr}
 \affil[4]{vinet@crm.umontreal.ca}
}
\date{\today}

\begin{document}
\maketitle

\hrule
\begin{abstract}
We compute the matrix elements of $SO(3)$ in any
finite-dimensional irreducible representation of $\sl_3$.
They are expressed in terms of a double sum of
products of Krawtchouk and Racah polynomials which generalize the Griffiths--Krawtchouk polynomials.
%, which are essentially
%the overlap coefficients in the case of symmetric representations of $\sl_3$.
Their recurrence and difference relations are obtained as byproducts of our construction.
The proof is based on the decomposition of a general three-dimensional rotation in terms
of elementary planar rotations and a transition between two embeddings of $\sl_2$
in $\sl_3$.
The former is related to monovariate Krawtchouk polynomials
and the latter, to monovariate Racah polynomials.
The appearance of Racah polynomials in this context is algebraically explained
by showing that the two $\sl_2$ Casimir elements related to the two embeddings of $\sl_2$
in $\sl_3$ obey the Racah algebra relations.
We also show that these two elements generate the centralizer in $U(sl_3)$ of the Cartan subalgebra
and its complete algebraic description is given.
\end{abstract}
\hrule

\section{Introduction}\label{sec:introduction}

Consider the particular irreducible representation of the Lie algebra $\sl_3$ given by the $n^\text{th}$ symmetric power of its three-dimensional defining representation.
The group $SO(3)$ acts naturally on this representation and its matrix elements are given (up to some normalization) by
the bivariate Griffiths--Krawtchouk polynomials. This was proven in \cite{GenestVinetetal2013b} by using oscillator algebras to realize the $n^\text{th}$ symmetric power representation.

The goal of this paper is to consider an arbitrary finite-dimensional irreducible representation of $\sl_3$. The group $SO(3)$ still acts naturally on this representation
and our main result is an expression of its matrix elements. They are expressed as
a novel family of $3$-variable functions, enjoying many nice properties: they are bispectral, orthogonal and are
given as sums of products of univariate Krawtchouk and Racah polynomials.
% It is shown that this setting provides a realization of the Racah algebra in $U(\sl_3)$
% as a byproduct.

\subsection*{On Krawtchouk and Racah polynomials}
The Krawtchouk and Racah polynomials have appeared in various aspects of representation
theory.
The univariate Krawtchouk polynomials appear as matrix elements of
irreducible representations of $SU(2)$ \cite{Koornwinder1982}.
They also appear as $3j$ and $6j$ coefficients of the oscillator algebra
\cite{VilenkinKlimyk1991, Vanderjeugt1997} and as spherical functions on wreath products
of symmetric groups \cite{Dunkl1976}.
Meanwhile, the Racah polynomials have been identified as the $6j$ coefficients of the
algebra $\sl_2$ \cite{Wilson1980, NikiforovUvarovetal1985} and an algebraic interpretation
has been given in \cite{GranovskiiZhedanov1988}.

These objects and their connections have been extended to higher dimensions in multiple
ways.
Griffiths introduced bivariate (and multivariate) Krawtchouk polynomials in 1971
\cite{Griffiths1971}. Twenty years later, Tratnik defined multivariate Racah polynomials
\cite{Tratnik1991} from which $d$-variable Krawtchouk polynomials could be obtained
through a limit.
The $d$-variable Griffiths--Krawtchouk polynomials depend on $\tfrac{1}{2}d(d+1)$
parameters while the $d$-variable Tratnik--Krawtchouk polynomials depend on $d$
parameters. Although one could expect that the Krawtchouk polynomials of Griffiths type
are more general than those of Tratnik type due to their larger number of parameters, the
precise relationship between the two remained unclear for quite some time. This is in part
due to the very different parametrizations of both families.
Things became clear in \cite{GenestVinetetal2013b} where an algebraic
interpretation cemented the fact that the $\tfrac{1}{2}d(d+1)$ parameters play the role of
the Euler angles that parametrize general $SO(d+1)$ rotations. The Tratnik--Krawtchouk
case could then be seen as a special case where some Euler angles vanish.
The Griffiths--Krawtchouk polynomials have also appeared in the study of
character algebras and been given an expression in terms of Aomoto--Gel'fand series
\cite{MizukawaTanaka2004}, the bivariate case was rediscovered from a probabilistic
perspective \cite{HoareRahman2008} and various algebraic interpretations of the bivariate
and multivariate Krawtchouk polynomials have been presented \cite{Rosenblyum1987,
Zhedanov1997a, IlievTerwilliger2012, Iliev2012, CrampevandeVijveretal2020}.
The multivariate Krawtchouk polynomials have also been studied in the context of birth and
death processes \cite{Milch1968}, Markov chains (see \cite{DiaconisGriffiths2014} for a
review) and their bispectrality has been established \cite{GeronimoIliev2010,
GenestVinetetal2013b}.  As for the multivariate Racah polynomials, they have appeared as
$3nj$ coefficients of the algebra $\sl_2$ \cite{DeBievandeVijver2020} and the bivariate
Racah polynomials have been studied in details in \cite{CrampeFrappatetal2022}.
We note that, so far, the Racah polynomials had never been given a direct
interpretation in terms of matrix elements of representations of the rotation group,
even if such a result could be deduced from \cite{ChaconMoshinsky1966, PluharSmirnovetal1985,
GelfandZelevinsky1986} where the transition coefficients between different Gelfand--Tsetlin bases are related to $6j$-symbols of $\sl_2$.
We use a more modern approach involving explicitly the Racah polynomials and hope that this will pave the way for a subsequent multivariate extension.

In this paper, the matrix elements of $SO(3)$ rotations in any representation of $\sl_3$ are
expressed as ``hybrid'' multivariate Krawtchouk and Racah polynomials.
Here is what we mean by ``hybrid''.
In the $d$-variable extensions of Tratnik (respectively Griffiths), the multivariate Racah
(Krawtchouk) polynomials can be expressed as sums of products of univariate Racah
(Krawtchouk) polynomials.
Meanwhile, the family that we obtain here is written as sums of products of both
univariate Racah and Krawtchouk polynomials.

Such hybrid multivariate polynomials appeared before (see for example \cite{Dunkl1976,
TarnanenAaltonenetal1985, GenestMikietal2014}), but they did not
attract that much attention.
Nevertheless, they have recently resurfaced as a byproduct of the definition of bivariate
$P$-polynomial association schemes \cite{BernardCrampeetal2023, CrampeVinetetal2023}.
Hybrid multivariate polynomial families seem to be more and more an important aspect of
future research in the field of special functions and their appearance in a very natural
setting in the present paper is an indication of that.

\subsection*{On the Racah algebra and bispectrality}
The Racah polynomials (like all families of the Askey scheme) possess a bispectrality
property \cite{Leonard1982,Grunbaum2001} as
they are simultaneously eigenfunctions of two eigenvalue problems in two
different variables.
The first eigenvalue problem is written as the three term recurrence relation in $n$, with
eigenvalue in $x$, and the second one as the difference equation in $x$, with eigenvalue
in $n$.
We refer to the recurrence and difference operators as the bispectral operators.

When both taken in the variable or degree representation, these operators do not commute
with each other and the algebraic relations that they obey are the so-called Racah algebra
relations. These read
\begin{align}\label{eq:racah3}
\begin{aligned}
 {}[K_2,[K_1,K_2]]&={K_2}^2+\{K_1,K_2\}+dK_2+e_1\,,\\
 [[K_1,K_2],K_1]&={K_1}^2+\{K_1,K_2\}+dK_1+e_2\,,
\end{aligned}
\end{align}
where $K_1$, $K_2$ are generators of this algebra and $d$, $e_1$ and $e_2$ are some
parameters (or central elements).

The Racah relations first appeared in the study of angular momentum recoupling
\cite{LevyLeblondLevyNahas1965} and were connected to the Racah polynomials in
\cite{GranovskiiZhedanov1988}.
Hence, whenever the Racah algebra appears, the Racah polynomials will be lurking, and vice-versa.
Remarkably, even though, the Racah algebra is quadratic, it can be studied in
detail. Its representation theory is well-known
\cite{HuangBocktingConrad2020,	% irreducible modules
HuangBocktingConrad2020a},	% Casimir elements
it has been embedded in various algebraic structures
\cite{GenestVinetetal2013,	% equitable su(1,1)
GenestVinetetal2015,		% embed in BI algebra
CorreaDelolmoetal2021,		% Racah in U(sl_3)
TerwilligerVidunas2004,		% Leonard pair AW relations
Koornwinder2007a,		% AW in DAHA
CrampePoulaindAndecyetal2019,	% Racah and centralizer
BocktingConradHuang2020,	% in equitable sl(2)
GaboriaudVinetetal2018} and it encapsulated the properties of the eponymous polynomials.	% Howe duality picture
It has been connected to physics and plays in particular an important role as the
symmetry algebra of the generic superintegrable system on the $2$-sphere
\cite{KalninsMilleretal2007} (see also \cite{GenestVinetetal2014} for a review).
Generalizations of this algebra to describe multivariate Racah polynomials lead to the
higher rank Racah algebras, which have also been extensively studied
\cite{
KalninsMilleretal2011,		% superintegrable system 3-sphere
Post2015,			% bivariate Racah
DeBieGenestetal2017,		% Laplace-Dunkl higher rank
DeBievandeVijver2020,		% discrete higher rank
CrampevandeVijveretal2020,	% with oscillator
CrampeGaboriaudetal2021,	% Racah as centralizer + HP series
CampoamorstursbergLatinietal2023}.	% rank 2 Racah in U(sl_4)...

The Racah polynomials will appear here in a new fashion as matrix elements of some
particular rotation in $sl_3$ representations. From the above, one would expect the
existence of a realization of the Racah algebra in $U(\sl_3)$, and indeed we exhibit
explicitly such a realization.

\subsection*{Outline}
We first define the algebra $U(\sl_3)$ and introduce its Gelfand--Tsetlin
basis in Section \ref{sec:usl3_gt}.
In Section \ref{sec:problem}, we pose the main problem precisely.
Inner automorphisms of $U(\sl_3)$ corresponding to $SO(3)$ rotations are introduced
and the strategy to compute the matrix elements by decomposing the general
rotation as a product of two types of rotations is explained.
In Section \ref{sec:Rz_krawtchouk}, we look at a first type of rotation, from which we
extract Krawtchouk polynomials.
A second type of rotation  analyzed in Section \ref{sec:T_racah} leads to Racah
polynomials.
Explicit expressions for the matrix elements associated to a generic rotation are
presented in Section \ref{sec:Rgeneral}.
Special cases of interest are displayed in Section \ref{sec:examples}.
In Section \ref{sec:Racah-algebra}, the realization of the Racah algebra in
$U(\sl_3)$ is presented, its connection with a centralizer is explained and its
Hilbert--Poincar\'e series is computed.
Closing remarks and perspectives conclude the paper.

\section{The algebra $U(sl_3)$ and its Gelfand--Tsetlin basis}\label{sec:usl3_gt}

Let us first introduce the algebra $U(\sl_3)$ and its finite-dimensional
irreducible representations in the Gelfand--Tsetlin bases.

\paragraph{Definition.} The enveloping algebra  $U(sl_3)$ of the Lie algebra $sl_3$ is
generated by the elements $e_{ij}$ with $1\leq i,j \leq 3$ satisfying the defining relations
\begin{align}
 [e_{ij}, e_{k\ell}] = \delta_{jk}e_{i\ell}  - \delta_{i\ell}e_{jk}\,,\qquad
 \sum_{i=1}^3 e_{ii}=0 \,.
\end{align}
The following Casimir elements generate the center of $U(sl_3)$:
\begin{equation} \label{CEsl3}
C_2=\sum_{i,j=1}^3 e_{ij}e_{ji}\qquad \text{and}\qquad C_3=\sum_{i,j,k=1}^3
e_{ij}e_{jk}e_{ki}\,.
\end{equation}
We shall take $J$ to be the Casimir element of the $sl_2$ subalgebra generated by
$e_{11}-e_{22}$, $e_{12}$ and $e_{21}$:
\begin{equation}\label{eq:J}
 J=\frac{(e_{11}-e_{22})^2+2(e_{11}-e_{22})}{4}+e_{21}e_{12}\, .
\end{equation}

\paragraph{Finite-dimensional irreducible representations.}
Finite-dimensional irreducible representations of $sl_3$ are in one-to-one correspondence
with $3$-tuples of complex numbers $\lambda=(\lambda_{31}, \lambda_{32},\lambda_{33})$,
called the highest weight, such that
$\lambda_{31}- \lambda_{32}\in\mathbb{Z}_{\geq 0}$,
$\lambda_{32}-\lambda_{33}\in\mathbb{Z}_{\geq 0}$ and $\lambda_{31}+
\lambda_{32}+\lambda_{33}=0$.
The representation associated to the highest weight $\lambda$ contains a unique (up to
normalization) non-zero vector $\xi$, called the highest weight vector, such that
\begin{equation}
 e_{ii}\, \xi = \lambda_{3,i}\, \xi \qquad
\text{for } 1\leq i \leq 3\,,  \qquad\text{and} \qquad  e_{ij}\, \xi =0\qquad
\text{for } 1\leq i<j \leq 3\, .
\end{equation}
This representation is also characterized by a two-row Young tableau with
$\lambda_{31}-\lambda_{33}$ boxes in the first row and $\lambda_{32}-\lambda_{33}$ boxes
in the second row.

This representation can also be described by a Gelfand--Tsetlin (GT) pattern which is
given by $\Lambda=( \lambda_{11}, \lambda_{21} , \lambda_{22}; \lambda_{31}, \lambda_{32},
\lambda_{33} )$ with the conditions
\begin{align}
\lambda_{31} - \lambda_{21},\quad \lambda_{21}- \lambda_{32},\quad \lambda_{32} -
\lambda_{22},\quad \lambda_{22}-  \lambda_{33},\quad  \lambda_{21}-   \lambda_{11},\quad
\lambda_{11}-\lambda_{22}\ \in\mathbb{Z}_{\geq 0}\,.
\end{align}
In the following the three last numbers in $\Lambda$ are fixed and we write only
$\Lambda=( \lambda_{11}, \lambda_{21} , \lambda_{22} )$.
The set of GT patterns for this highest weight $\lambda$ is denoted $\mathcal P_\lambda$.
For a GT pattern $\Lambda$, one associates the representation basis vectors, called GT
vectors (see \cite{Molev2006} and references therein)\footnote{
For later convenience, we change the normalization of the vectors in comparison with
\cite{Molev2006}:\\
$\xi_{\lambda\ old}\to
\frac{(\lambda_{21}-\lambda_{32})!(\lambda_{21}-\lambda_{33}+1)!
(\lambda_{22}-\lambda_{33})!}{(\lambda_{32}-\lambda_{22})!}\xi_\lambda$.}
\begin{equation}
 \xi_\Lambda= \left| \begin{array}{ccccc}
 \lambda_{31} & & \lambda_{32} & & \lambda_{33} \\
 & \lambda_{21} & & \lambda_{22} & \\
 && \lambda_{11}\end{array}\right\rangle.
\end{equation}
The $sl_3$ generators $e_{ij}$ act as follows on these vectors:
\begin{subequations}\label{eq:actGT}
 \begin{align}
 &e_{11}\xi_\Lambda= \lambda_{11} \xi_\Lambda \,, \qquad
  e_{22}\xi_\Lambda= (\lambda_{21} +\lambda_{22}- \lambda_{11} ) \xi_\Lambda\,, \qquad
  e_{33}\xi_\Lambda= -(\lambda_{21} +\lambda_{22} )  \xi_\Lambda \,,\\
   &e_{12}\xi_\Lambda = (\lambda_{21}-\lambda_{11}) (\lambda_{11}-\lambda_{22}+1) \xi_{\Lambda+\delta^{11}}\,,\qquad
     e_{21}\xi_\Lambda =\xi_{\Lambda-\delta^{11}}\,,\\
     &e_{23}\xi_\Lambda =\frac{\lambda_{31}-\lambda_{21}   }{\lambda_{21}-\lambda_{22}+1}\xi_{\Lambda+\delta^{21}}
     +\frac{ \lambda_{31}-\lambda_{22}+1}{\lambda_{21}-\lambda_{22}+1}\xi_{\Lambda+\delta^{22}}\,,\\
      &e_{32}\xi_\Lambda =\frac{(\lambda_{21}-\lambda_{32})(\lambda_{21}-\lambda_{33}+1)(\lambda_{21}-\lambda_{11})   }{\lambda_{21}-\lambda_{22}+1}\xi_{\Lambda-\delta^{21}}\nonumber\\
   &\qquad \qquad +\frac{(\lambda_{11}-\lambda_{22}+1)(\lambda_{32}-\lambda_{22} +1)(\lambda_{22}-\lambda_{33})  }{\lambda_{21}-\lambda_{22}+1}\xi_{\Lambda-\delta^{22}}\,,
\end{align}
where  $\xi_{\Lambda \pm \delta^{ij}}$ is either the basis
element associated to the GT pattern $\Lambda \pm \delta^{ij}$
where the value of $\lambda_{ij}$ has become $\lambda_{ij} \pm 1$, or $0$
if the resulting pattern is not a valid GT pattern.
One can deduce the actions of the remaining generators:
\begin{align}
  &e_{13}\xi_\Lambda =\frac{(\lambda_{11}-\lambda_{22}+1)(\lambda_{31}-\lambda_{21})}
 {\lambda_{21}-\lambda_{22}+1}\xi_{\Lambda+\delta^{21}+\delta^{11}}
  -\frac{(\lambda_{21}-\lambda_{11})( \lambda_{31}-\lambda_{22}+1)}
 {\lambda_{21}-\lambda_{22}+1}\xi_{\Lambda+\delta^{22}+\delta^{11}}\,,\\
  &e_{31}\xi_\Lambda =\frac{(\lambda_{21}-\lambda_{32})(\lambda_{21}-\lambda_{33}+1)}
 {\lambda_{21}-\lambda_{22}+1}\xi_{\Lambda-\delta^{21}-\delta^{11}}
  -\frac{(\lambda_{32}-\lambda_{22} +1)(\lambda_{22}-\lambda_{33})}
 {\lambda_{21}-\lambda_{22}+1}\xi_{\Lambda-\delta^{22}-\delta^{11}}\,.
\end{align}
\end{subequations}

The Casimir elements of $U(sl_3)$ are proportional to the identity in this representation
and, using \eqref{eq:actGT}, one gets
\begin{subequations}
\begin{align}\label{Casimir-spec}
&C_2\, \xi_\Lambda= 2 (\lambda_{31}^2+\lambda_{31}\lambda_{32}
 + \lambda_{32}^2+2 \lambda_{31}+\lambda_{32}) \xi_\Lambda\,,\\
&C_3\, \xi_\Lambda=3 \lambda_{31} (1 - \lambda_{32})
 (2 + \lambda_{31} + \lambda_{32}) \xi_\Lambda\,.
\end{align}
\end{subequations}
The element $J$ defined by \eqref{eq:J} acts diagonally on this basis
\begin{equation}
 J\, \xi_\Lambda= \frac{1}{4}(\lambda_{21}-\lambda_{22})(\lambda_{21}-\lambda_{22}+2) \, \xi_\Lambda\,.
\end{equation}
In fact,  this eigenvalue of the $\sl_2$ Casimir element $J$, with the ones  of the hypercharge $Y=\frac{1}{3}(e_{11}+e_{22}-2e_{33})$ and the
Cartan element $H=e_{11}-e_{22}$:
\begin{align}
 &Y\, \xi_\Lambda= (\lambda_{21}+\lambda_{22}) \, \xi_\Lambda\,,\\
 &H\, \xi_\Lambda=(2\lambda_{11}-\lambda_{21}-\lambda_{22}) \, \xi_\Lambda\,,
\end{align}
completely
characterize the vectors $ \xi_\Lambda$ (if the highest weight $\lambda$ is given).

Let us introduce the normalized vectors
\begin{equation}\label{eq:zeta}
 \zeta_\Lambda= \frac{1}{N_\Lambda} \xi_\Lambda\,,
\end{equation}
with
\begin{align}
 \left( N_\Lambda\right)^2&=\frac{1}{\lambda_{21}-\lambda_{22}+1} (\lambda_{21}-\lambda_{11})!(\lambda_{31}-\lambda_{21})!(\lambda_{31}-\lambda_{22}+1)!(\lambda_{31}-\lambda_{32})!(\lambda_{32}-\lambda_{33})! \nonumber\\
 &\times\frac{(\lambda_{22}-\lambda_{33})!(\lambda_{31}-\lambda_{33}+1)!(\lambda_{21}-\lambda_{32})!(\lambda_{21}-\lambda_{33}+1)!}
 {(\lambda_{11}-\lambda_{22})!(\lambda_{32}-\lambda_{22})!}.
\end{align}
In this basis, the anti-automorphism $\star$ of $sl_3$ defined by $\star: e_{ij} \mapsto
e_{ji}$ corresponds to the transposition.

\section{The general problem: rotations and change of basis coefficients}
\label{sec:problem}

In this Section, we introduce the inner automorphisms of $\sl_3$ associated to elements of the group $SO(3)$.
The matrix elements of $SO(3)$ are interpreted  as the overlap coefficients between different bases related by these inner automorphisms.
We provide the decomposition of a general rotation of $SO(3)$ into a product of two types
of rotations $R^{z}$ and $T$, which will be examined separately in the following two
Sections.

\paragraph{Inner automorphism.}
Let $R\in SO(3)$ be a $3\times 3$-matrix whose real entries satisfy
\begin{equation}
 \sum_{j=1}^3 R_{ij} R_{kj} =\sum_{j=1}^3 R_{ji} R_{jk} = \delta_{ik} \,.
\end{equation}
The map $\Psi_R$, labeled by $R\in SO(3)$ and acting on $g\in sl_3$ by $\Psi_R(g)=RgR^{-1}$, provides an automorphism of $U(sl_3)$ given explicitly by:
\begin{equation}
 \Psi_R\ :\ e_{ij} \mapsto  \sum_{k,\ell=1}^3 R_{ki}R_{\ell j}e_{k\ell}\,.
\end{equation}
One gets that $ \Psi_R \circ \Psi_{R'}= \Psi_{RR'}$ for any $R,R'\in SO(3)$.

From now on, let us fix a highest weight $\lambda=(\lambda_{31},
\lambda_{32},\lambda_{33})$ and let $\xi_\Lambda$ be vectors of the representation space.
For any rotation $R$, we denote $\rho$ the operator representing $R^{-1}$. It satisfies:
\begin{equation}\label{eq:fond1}
 \Psi_R(g)  \cdot  \xi_\Lambda = \rho^{-1}  \cdot  g  \cdot  \rho  \cdot  \xi_\Lambda\,\qquad \text{for any $g \in sl_3$} \, ,
\end{equation}
where
\begin{equation}
 \rho \cdot \xi_\Lambda =\sum_{\Lambda'\in \mathcal{P}_\lambda} \rho_{\Lambda',\Lambda} \, \xi_{\Lambda'}\,.
\end{equation}
This means that $\Psi_R(g)$ and $g$ are two equivalent representations and the change of
basis is provided by $\rho$. The coefficients $\rho_{\Lambda',\Lambda}$ are
the matrix elements of $R^{-1}$ in the Gelfand--Tsetlin basis.

In the basis of normalized vectors $\zeta_\Lambda$ \eqref{eq:zeta}, the matrices
corresponding to $\rho$ are orthogonal.
Therefore, due to the change of normalization, one gets the following orthogonality
relation
\begin{equation} \label{eq:ortho}
\sum_{\Lambda'\in \mathcal{P}_\lambda} \left(N_{\Lambda'}\right)^2\,  \rho_{\Lambda',\Lambda}\, \rho_{\Lambda',\Lambda''} =(N_{\Lambda})^2\, \delta_{\Lambda,\Lambda''}\,.
\end{equation}

\paragraph{Change of basis for any rotation.}
The goal of this paper is to provide an explicit formula for the entries of $\rho$ for any
$R$ and any representation.
When the representation is a symmetric representation of $U(sl_3)$ (\textit{i.e.}
$\lambda_{32}=\lambda_{33}$), the entries of $\rho$ are given in terms of
the Griffiths--Krawtchouk polynomials \cite{IlievTerwilliger2012, GenestVinetetal2013b}.
As explained in the introduction, we generalize this result to any finite-dimensional
irreducible representation of $\sl_3$.

Following \cite{GenestVinetetal2013b}, to compute the change of basis $\sigma$ for any $S
\in SO(3)$, we decompose $S$ into three elementary rotations: two rotations around the $z$
axis and one around the $y$ axis.
\begin{equation}\label{eq:S}
 {S}=R^{z}_\chi R^{y}_\theta  R^{z}_\phi\,,
 \end{equation}
 with
 \begin{equation}
R^z_\phi=\begin{pmatrix}
                                  \cos(\phi) & \sin(\phi) &0\\
                                  -\sin(\phi)& \cos(\phi)&0\\
                                  0&0 &1
                                 \end{pmatrix},  \qquad\qquad
  R^y_\theta=\begin{pmatrix}
                                  \cos(\theta) & 0 & -\sin(\theta) \\
                                  0& 1 & 0 \\
                                 \sin(\theta)& 0 & \cos(\theta)
                                 \end{pmatrix}\,.
 \end{equation}
The rotations around the $z$ axis are easy to deal with since they leave the Casimir
element  $J$ of $sl_2$ invariant and the calculation is done in Section \ref{sec:Rz_krawtchouk}.
For the rotation around the $y$ axis, the computation is more involved but the idea
consists in writing this rotation as follows:
\begin{equation}
 R^y_\theta= T R^z_{\theta} T^{-1}\,,
\end{equation}
where
\begin{equation}
 T=\begin{pmatrix}
  1 & 0 &0 \\
   0& 0&1\\
   0 &-1 &0
  \end{pmatrix}\,.
\end{equation}
The change of basis $\tau$ associated to the transformation $T$ is computed in Section
\ref{sec:T_racah} in terms of Racah polynomials.

After the changes of basis $\rho_\phi$ and $\tau$ associated to $R^z_\phi$ and $T$ have
been computed, one can write the whole change of basis $\sigma=\rho_\phi\, \tau^{-1}\,
\rho_{\theta}\, \tau\, \rho_\chi$ associated to the rotation ${S}=R^{z}_\chi T R^z_{\theta} T^{-1}   R^{z}_\phi\,,$ since
\begin{align}
   \Psi_{S}(g)  \cdot  \xi_\Lambda &=\Psi_{R^z_\chi T R^z_\theta T^{-1} R^z_\phi} (g)  \cdot  \xi_\Lambda  =\Psi_{R^z_\chi}\Bigg(\Psi_{T}\Big(\Psi_{R^z_{\theta}}\big(\Psi_{T^{-1}}(\Psi_{R^z_\phi} (g))\big)\Big)\Bigg)  \cdot  \xi_\Lambda \nonumber\\
&=  (\rho_\phi\, \tau^{-1} \, \rho_{\theta}\, \tau\, \rho_\chi)^{-1}\, \cdot g \cdot (\rho_\phi\, \tau^{-1} \, \rho_{\theta}\, \tau\, \rho_\chi) \cdot \xi_\Lambda \, .
\end{align}
We shall show that this full transformation is given as a double sum of a product of three Krawtchouk
and two Racah polynomials (see Proposition \ref{pr:res}).

\section{The change of basis for $R^z_\phi$ and the Krawtchouk polynomials}
\label{sec:Rz_krawtchouk}
In this section, we focus on the case where $R=R^z_\phi$ and we denote the change of basis
by $\rho$.
Imposing relation \eqref{eq:fond1} for $g=Y$, $H$ and $J$ constrains the operator $\rho$.
Let us remark that $\Psi_{R}(Y)=Y$ and $\Psi_{R}(J)=J$, therefore relation
\eqref{eq:fond1} reduces to
\begin{equation}
 Y\cdot \rho \cdot \xi_\Lambda =  \rho \cdot Y\cdot \xi_\Lambda\qquad \text{and}\qquad
 J\cdot \rho\cdot \xi_\Lambda =  \rho \cdot J \cdot \xi_\Lambda\,.
\end{equation}
Using the explicit expressions of the action of $sl_3$ on the GT basis, one gets
\begin{align}
\begin{aligned}
 0&=\rho_{\Lambda'\Lambda}(\lambda'_{21}+\lambda'_{22}-\lambda_{21}-\lambda_{22})\,, \\
 0&=\rho_{\Lambda'\Lambda}( (\lambda'_{21}-\lambda'_{22})(\lambda'_{21}-\lambda'_{22}+2)
    - (\lambda_{21}-\lambda_{22})(\lambda_{21}-\lambda_{22}+2)  )\,,
\end{aligned}
\end{align}
which imply that
\begin{equation}\label{eq:res1}
 \rho_{\Lambda'\Lambda}= \delta_{\lambda'_{21},\lambda_{21}}\delta_{\lambda'_{22},\lambda_{22}}\  r(\lambda_{11},\lambda'_{21},\lambda'_{11},\lambda'_{22}) \, .
\end{equation}
One also gets
\begin{equation}
\Psi_{R}(H)=\left(\cos^2(\phi)- \sin^2(\phi)\right) H
 -2\cos(\phi)\sin(\phi)( e_{12} +e_{21} )\,.
\end{equation}
Relation \eqref{eq:fond1} for $g=H$ provides the following constraint (and using relation
\eqref{eq:res1}):
\begin{align}
 (2\lambda'_{11}-\lambda'_{21}-\lambda'_{22}) \rho_{\Lambda',\Lambda}=&
 -2\cos(\phi)\sin(\phi)(\lambda_{21}-\lambda_{11})(\lambda_{11}-\lambda_{22}+1)\rho_{\Lambda',\Lambda+\delta^{11}}\nonumber \\
 &+(\cos^2(\phi)- \sin^2(\phi)) (2\lambda_{11}-\lambda_{21}-\lambda_{22}) \rho_{\Lambda',\Lambda}\nonumber \\
 &-2\cos(\phi)\sin(\phi) \rho_{\Lambda',\Lambda-\delta^{11}} \,. \label{eq:KK}
\end{align}
This relation is similar to the recurrence relation of the Krawtchouk polynomials
$K_n(x;p,N)$:
\begin{align}\label{}
 K_n(x;p,N)=
\pFq{2}{1}{-n\,,-x}{-N}{\frac{1}{p}}\,,
\end{align}
where ${}_2F{}_1$ is the usual hypergeometric function \cite{KoekoekLeskyetal2010}.
We use the convention that  $K_n(x;p,N)=0$ if $n,x<0$ or $n,x>N$.
By direct check, we see that
\begin{equation}
\label{eq:eqn28}
 r(\lambda_{11},\lambda'_{21},\lambda'_{11},\lambda'_{22})= \frac{\tan^{\lambda_{11}-\lambda'_{22}}(\phi) }{(\lambda_{11}-\lambda'_{22})!} \, K_{\lambda_{11}-\lambda'_{22}}\Big(\lambda'_{11}-\lambda'_{22};\sin^2(\phi),\lambda'_{21}-\lambda'_{22}\Big)
 r(\lambda'_{21},\lambda'_{11},\lambda'_{22})
\end{equation}
is the solution of \eqref{eq:KK}.
Here the indices $\lambda_{11}-\lambda_{22}'$ and variables $\lambda_{11}'-\lambda_{22}'$
both go from $0$ to $\lambda_{21}'-\lambda_{22}'$.

It remains to determine the function $r$. By using the orthogonality relation
\eqref{eq:ortho} satisfied by $\rho$, one can show
\begin{align}
\label{eq:norm2_krawtchouk}
 r(x,y,z)^{2}=\left(\frac{(x-z)!}{(x-y)!} \tan^{y-z}(\phi)\cos^{x-z}(\phi)\right)^{2}
\end{align}
by making use of the orthogonality relation of the Krawtchouk polynomials
\cite{KoekoekLeskyetal2010}
\begin{align}
 \binom{N}{n} \sum_{X=0}^N \binom{N}{X}  p^{X+n}(1-p)^{N-X-n}K_m(X;p,N)K_n(X;p,N)
 = \delta_{n,m}\,.
\end{align}
The previous results lead to the following proposition.
\begin{prop}
 The change of basis $\rho$ associated to $R^z_\phi$ is given by
 \begin{align}\label{eq:rho}
  \rho_{\Lambda'\Lambda}=& \delta_{\lambda'_{21},\lambda_{21}}
 \delta_{\lambda'_{22},\lambda_{22}}\
 (-1)^{\lambda_{11}'-\lambda_{22}}\
  \frac{(\lambda_{21}-\lambda_{22})!\tan^{\lambda'_{11}+\lambda_{11}-2\lambda_{22}}(\phi)  \cos^{\lambda_{21}-\lambda_{22}}(\phi)}{(\lambda_{11}-\lambda_{22})!(\lambda_{21}-\lambda'_{11})!} \nonumber\\
  &\times K_{\lambda_{11}-\lambda_{22}}\Big(\lambda'_{11}-\lambda_{22};\sin^2(\phi),\lambda_{21}-\lambda_{22}\Big) \,.
 \end{align}
\end{prop}
\begin{proof}
By combining \eqref{eq:eqn28} and taking the square root of \eqref{eq:norm2_krawtchouk},
we can obtain the change of basis $\rho$ up to signs.
The signs may be fixed by using equation \eqref{eq:fond1} for the other elements of
$sl_3$ but it is simpler to remark that
the operator $\rho$ tends continuously to the identity operator when $R$ tends to the
identity (\textit{i.e.} $\phi\to 0$):
\begin{align}\label{}
\rho_{\Lambda',\Lambda}\Big|_{\phi=0}=
\delta_{\lambda'_{21},\lambda_{21}}
\delta_{\lambda'_{11},\lambda_{11}}
\delta_{\lambda'_{22},\lambda_{22}}\,.
\end{align}
Using
\begin{align}\label{}
 p^{(x+n)/2}\binom{N}{n}K_{n}(x;p,N)\Big|_{p=0}=(-1)^{x}\delta_{x,n}\,,
\end{align}
the result \eqref{eq:rho} follows.
\end{proof}

\section{The change of basis corresponding to $T$ and the Racah\\ polynomials}
\label{sec:T_racah}
Under the $\sl_3$ automorphism associated to the rotation $T$,
the $\sl_2$ subalgebra on indices $1,2$ is sent to the subalgebra on indices $1,3$.
Indeed, after a short calculation, one can show that the $\sl_2$ Casimir element $J$
is mapped to
\begin{equation}\label{eq:TJ}
 \Psi_T(J)=\frac{(e_{11}-e_{33})^2+2(e_{11}-e_{33})}{4}+e_{31}e_{13}\,,
\end{equation}
which is precisely the Casimir element of the subalgebra $sl_2$ generated by
$e_{11}-e_{33}$, $e_{13}$ and $e_{31}$.

Relations \eqref{eq:fond1} for $R=T$ and $g=H,Y$ provide the following two constraints on
the entries of the change of basis $\tau$:
\begin{subequations} \label{eq:11}
 \begin{align}
0&=(\lambda'_{11}-\lambda_{11})\tau_{\Lambda'\Lambda}\,, \label{eq:l11}\\
0&=(\lambda_{21}+\lambda_{22}-\lambda'_{11} +\lambda'_{21}+\lambda'_{22})
 \tau_{\Lambda'\Lambda}\,.
\end{align}
\end{subequations}
These constraints lead to the following form for $\tau$:
\begin{equation}\label{eq:tau4t}
 \tau_{\Lambda' \Lambda}= \delta_{\lambda'_{11}, \lambda_{11}}\delta_{\lambda_{21}+\lambda_{22}, \lambda'_{11}- \lambda'_{21}-\lambda'_{22}}\, t(\lambda_{21},\lambda'_{21},\lambda'_{11},\lambda'_{22})\,.
\end{equation}
Similarly, equation \eqref{eq:fond1} for $g=J$ provides the following constraint
(where we use relations \eqref{eq:11} to replace $\lambda_{11}$ by $\lambda'_{11}$ and
$\lambda_{21}+\lambda_{22}$ by $\lambda'_{11}-\lambda'_{21}-\lambda'_{22}$):
\begin{align}\label{eq:recN}
& A_\Lambda\, \tau_{\Lambda',\Lambda-\delta^{21}+\delta^{22}} +\left(A_\Lambda+C_\Lambda \right) \tau_{\Lambda'\Lambda}+ C_\Lambda \tau_{\Lambda',\Lambda+\delta^{21}-\delta^{22}}
=(\lambda'_{21} -\lambda_{31})(\lambda'_{22} -\lambda_{31}-1)\tau_{\Lambda'\Lambda}\,,
\end{align}
with
\begin{subequations}
 \begin{align}
 &A_\Lambda=\frac{(\lambda_{21}-\lambda_{32})(\lambda_{31}-\lambda_{22}+1)(\lambda_{21}-\lambda_{11})(\lambda_{21}-\lambda_{33}+1)}{(\lambda_{21}-\lambda_{22}+1)(\lambda_{21}-\lambda_{22})}\,,\\
 &C_\Lambda=\frac{(\lambda_{11}-\lambda_{22}+1)(\lambda_{22}-\lambda_{33})(\lambda_{32}-\lambda_{22}+1)(\lambda_{31}-\lambda_{21})}{(\lambda_{21}-\lambda_{22}+1)(\lambda_{21}-\lambda_{22}+2)}\,.
\end{align}
\end{subequations}
Relation \eqref{eq:recN} is the recurrence relation of the Racah polynomial
\cite{KoekoekLeskyetal2010}.
% $R_{n}(x(x+\gamma+\delta+1) ;\alpha,\beta,\gamma,\delta)$
\begin{align}\label{}
 % \widetilde R_{n}(x ;\alpha,\beta,\gamma,\delta)=
 R_n(x(x+\gamma+\delta+1) ;\alpha,\beta,\gamma,\delta)=
\pFq{4}{3}{-n\,,n+\alpha+\beta+1\,,-x\,,x+\gamma+\delta+1}{\alpha+1\,,\beta+\delta+1\,,
\gamma+1}{1}\,.
\end{align}
To simplify the notation, one defines
$\widetilde R_{n}(x ;\alpha,\beta,\gamma,\delta)=R_{n}(x(x+\gamma+\delta+1) ;\alpha,\beta,\gamma,\delta)$.
We use also the convention that $\widetilde R_{n}(x ;\alpha,\beta,\gamma,\delta)=0$ if $n,x<0$ or $n,x>N$ where $N=min(-\alpha-1,-\beta-\delta-1,-\gamma-1)$.
One gets explicitly:
\begin{subequations}\label{eq:4t3t}
\begin{align}
&t(\lambda_{21},\lambda'_{21},\lambda'_{11},\lambda'_{22})=t(\lambda'_{21},\lambda'_{11},\lambda'_{22})\ (-1)^{\lambda_{31}-\lambda_{21}}\
\widetilde{R}_{\lambda_{31}-\lambda_{21}}( \lambda_{31}-\lambda'_{21}; \alpha,\beta,\gamma,\delta) \,,
\end{align}
with
 \begin{align}
 &\alpha= \lambda_{32}-\lambda_{31}-1\,,\qquad\quad \beta=\lambda'_{11}-\lambda'_{21}-\lambda'_{22} +\lambda_{33}-1\,,\\
  &\gamma= \lambda'_{11}-\lambda_{31}-1\,,\qquad\quad \delta= \lambda'_{21}+\lambda'_{22}-\lambda'_{11} -\lambda_{31}-1\,.
\end{align}
\end{subequations}
The degree $\lambda_{31}-\lambda_{21}$ of the Racah polynomials is in the set
$\{0,1,\dots, \text{min}(-\alpha_{\Lambda}-1,-\gamma_{\Lambda}-1 ) \}$ and its variable
$\lambda_{31}-\lambda'_{21}$ is in the set $\{0,1,\dots,
\text{min}(-\alpha_{\Lambda'}-1,-\gamma_{\Lambda'}-1 ) \}$. After the identification
\eqref{eq:11}, both sets are identical.

As in the previous section, the factor $t(\lambda'_{21},\lambda'_{11},\lambda'_{22})$ is
determined by using the orthogonality relation satisfied by $\tau$ and the one of the
Racah polynomials \cite{KoekoekLeskyetal2010}.
After a few manipulations on the parameters, one obtains
\begin{equation}\label{eq:t}
 t(x,y,z)^{2}
 =\left((x-z+1)\,\frac{(\lambda_{31}-\lambda_{32})!(\lambda_{31}-\lambda_{33}+1)!
 (\lambda_{31}-y)!(\lambda_{32}-z)!(y-z)!}
 {(x-\lambda_{32})!(x-\lambda_{33}+1)!(x-y)!(\lambda_{31}-z+1)!(\lambda_{31}-x)!
 (z-\lambda_{33})!}\right)^{2}\,.
\end{equation}
The previous results lead to the following proposition
\begin{prop}
The matrix elements of the change of basis $\tau$ corresponding to the rotation $T$ are given (up to a
global undetermined sign) by
 \begin{align}\label{eq:tau}
  \tau_{\Lambda'\Lambda}=&\delta_{\lambda'_{11}, \lambda_{11}}\delta_{\lambda_{21}+\lambda_{22}, \lambda'_{11}- \lambda'_{21}-\lambda'_{22}}\,
  t(\lambda'_{21},\lambda'_{11},\lambda'_{22})(-1)^{\lambda'_{22}-\lambda_{21}} \\
 \times  &
  \widetilde{R}_{\lambda_{31}-\lambda_{21}}(\lambda_{31}-\lambda'_{21}  ;  \lambda_{32}-\lambda_{31}-1 ,\lambda_{21}+\lambda_{22} +\lambda_{33}-1 ,\lambda_{11}-\lambda_{31}-1 ,-\lambda_{21}-\lambda_{22}-\lambda_{31}-1)\,, \nonumber
 \end{align}
with
\begin{equation}\label{eq:tbis}
 t(x,y,z)
 =(x-z+1)\,\frac{(\lambda_{31}-\lambda_{32})!(\lambda_{31}-\lambda_{33}+1)!
 (\lambda_{31}-y)!(\lambda_{32}-z)!(y-z)!}
 {(x-\lambda_{32})!(x-\lambda_{33}+1)!(x-y)!(\lambda_{31}-z+1)!(\lambda_{31}-x)!
 (z-\lambda_{33})!}\,.
\end{equation}
\end{prop}
\proof
From equations \eqref{eq:tau4t}, \eqref{eq:4t3t}, \eqref{eq:t},
$\tau_{\Lambda',\Lambda}$
is determined up to a sign $\sigma_{\lambda'_{21}, \lambda'_{11}, \lambda'_{22}}$
which could depend on the values of $\lambda'_{21}$,
$\lambda'_{11}$, $\lambda'_{22}$:
\begin{subequations}\label{eq:tautemp}
\begin{align}
 \tau_{\Lambda'\Lambda}=&\delta_{\lambda'_{11}, \lambda_{11}}
 \delta_{\lambda_{21}+\lambda_{22}, \lambda'_{11}- \lambda'_{21}-\lambda'_{22}}\,
 t(\lambda'_{21},\lambda'_{11},\lambda'_{22})(-1)^{\lambda_{31}-\lambda_{21}}
 \widetilde{R}_{\lambda_{31}-\lambda_{21}}(\lambda_{31}-\lambda'_{21}  ;
  \alpha,\beta,\gamma,\delta)\,,
 \end{align}
with
\begin{align}\label{eq:signfct}
 t(\lambda'_{21}, \lambda'_{11}, \lambda'_{22})=
 T(\lambda'_{21}, \lambda'_{11}, \lambda'_{22})
 \sigma_{\lambda'_{21}, \lambda'_{11}, \lambda'_{22}}\,,
\end{align}
and
\begin{align}\label{}
 T(x,y,z)
 =(x-z+1)\,\frac{(\lambda_{31}-\lambda_{32})!(\lambda_{31}-\lambda_{33}+1)!
 (\lambda_{31}-y)!(\lambda_{32}-z)!(y-z)!}
 {(x-\lambda_{32})!(x-\lambda_{33}+1)!(x-y)!(\lambda_{31}-z+1)!(\lambda_{31}-x)!
 (z-\lambda_{33})!}\,.
\end{align}
\end{subequations}
Once the sign function has been determined, one immediately recovers the full expression
of $\tau_{\Lambda',\Lambda}$ from equations \eqref{eq:tau4t}, \eqref{eq:4t3t},
\eqref{eq:signfct}.

The expression for the sign function $\sigma_{\lambda'_{21}, \lambda'_{11},
\lambda'_{22}}$ is obtained from constraints following from \eqref{eq:fond1} for other elements of the algebra $U(\sl_3)$.

First, use $g=e_{12}$ in relation \eqref{eq:fond1} to get
\begin{align*}\label{}
0&=(\lambda_{21}-\lambda_{11})(\lambda_{11}-\lambda_{22}+1)
 \tau_{\Lambda',\Lambda+\delta^{11}}
-\frac{(\lambda'_{11}-\lambda'_{22})(\lambda'_{31}-\lambda'_{21}+1)}
 {(\lambda'_{21}-\lambda'_{22})}\tau_{\Lambda'-\delta^{21}-\delta^{11},\Lambda}\\
&+\frac{(\lambda'_{21}-\lambda'_{11}+1)(\lambda'_{31}-\lambda'_{22}+2)}
 {(\lambda'_{21}-\lambda'_{22}+2)}\tau_{\Lambda'-\delta^{22}-\delta^{11},\Lambda}\,.
\end{align*}
Substituting \eqref{eq:tautemp}, and using the following relation between Racah
polynomials \cite[equation~(4.13)]{Wilson1978}
\begin{align}\label{}
\begin{aligned}
(n+\gamma)(n-\gamma+\alpha+\beta+1)&~\widetilde{R}_{n}(x;\alpha,\beta,\gamma,\delta)\\
=\frac{\gamma}{2x+\gamma+\delta+1}&\Big[
(x+\alpha+1)(x+\beta+\delta+1)\widetilde{R}_{n}(x+1;\alpha,\beta,\gamma-1,\delta)\\
&
-(x-\alpha+\gamma+\delta)(x-\beta+\gamma)\widetilde{R}_{n}(x;\alpha,\beta,\gamma-1,\delta)
\Big]
\end{aligned}
\end{align}
allows one to fix
\begin{align}\label{eq:sign1}
 \sigma_{\lambda'_{21}-1, \lambda'_{11}-1, \lambda'_{22}}
 =(+1)\sigma_{\lambda'_{21}, \lambda'_{11}, \lambda'_{22}}\,,\qquad
 \sigma_{\lambda'_{21}, \lambda'_{11}-1, \lambda'_{22}-1}
 =(-1)\sigma_{\lambda'_{21}, \lambda'_{11}, \lambda'_{22}}\,.
\end{align}
Next, use $g=e_{23}$ in relation \eqref{eq:fond1}.
% to get
% \begin{align}\label{}
%  0=()\tau_{\Lambda',\Lambda+\delta^{21}}+()\tau_{\Lambda',\Lambda+\delta^{22}}
% +()\tau_{\Lambda'+\delta^{21},\Lambda}+()\tau_{\Lambda'+\delta^{22},\Lambda}\,.
% \end{align}
Substituting \eqref{eq:tautemp}, and using the following relation between
Racah polynomials\footnote{This relation follows from applying the following
hypergeometric contiguity relation
to simplify the first two terms and the last two terms of \eqref{eq:quatre_racah}
% a total of two times:
\begin{align}\label{}
 \frac{\alpha_1}{\alpha_2-\alpha_1}
 \pFq{4}{3}{\alpha_1+1\,,\alpha_2\,,\alpha_3\,,\alpha_4}{\beta_1\,,\beta_2\,,\beta_3}{z}
-\frac{\alpha_2}{\alpha_2-\alpha_1}
 \pFq{4}{3}{\alpha_1\,,\alpha_2+1\,,\alpha_3\,,\alpha_4}{\beta_1\,,\beta_2\,,\beta_3}{z}
=\pFq{4}{3}{\alpha_1\,,\alpha_2\,,\alpha_3\,,\alpha_4}{\beta_1\,,\beta_2\,,\beta_3}{z}
\end{align}
}
\begin{align}\label{eq:quatre_racah}
\begin{aligned}
0&=\frac{n}{2n+\alpha+\beta}\widetilde{R}_{n-1}(x;\alpha,\beta,\gamma,\delta)
 +\frac{n+\alpha+\beta}{2n+\alpha+\beta}\widetilde{R}_{n}(x;\alpha,\beta,\gamma,\delta)\\
 &-\frac{x}{2x+\gamma+\delta+1}\widetilde{R}_{n}(x-1;\alpha,\beta-1,\gamma,\delta+1)
 -\frac{x+\gamma+\delta+1}{2x+\gamma+\delta+1}
  \widetilde{R}_{n}(x;\alpha,\beta-1,\gamma,\delta+1)
\end{aligned}
\end{align}
allows one to fix
\begin{align}\label{eq:sign2}
 \sigma_{\lambda'_{21}+1, \lambda'_{11}, \lambda'_{22}}
 =(+1)\sigma_{\lambda'_{21}, \lambda'_{11}, \lambda'_{22}}\,,\qquad
 \sigma_{\lambda'_{21}, \lambda'_{11}, \lambda'_{22}+1}
 =(-1)\sigma_{\lambda'_{21}, \lambda'_{11}, \lambda'_{22}}\,.
\end{align}
% \BLUE{
% Finally, replace $g$ by $\Psi_{S^t}(g)$ in relation \eqref{eq:fond1} to get
% \begin{equation}\label{eq:fonS2}
%  \sigma\cdot g  \cdot  \xi_\Lambda =  \Psi_{S^t}(g)  \cdot  \sigma  \cdot  \xi_\Lambda\,,
% \end{equation}
% where $S^t$ stands for the transposition of $S$ and we have used that
% $\Psi_{S}(\Psi_{S^t}(g)) = g$.
% Then, taking $g=J$ and substituting \eqref{eq:tautemp}, one recognizes the difference
% equation for the Racah polynomials.
% Comparing with the exact expression found in \cite{KoekoekLeskyetal2010} allows one to fix
% \begin{align}\label{eq:sign3}
%  \sigma_{\lambda'_{21}-1, \lambda'_{11}, \lambda'_{22}+1}
%  =(-1)\sigma_{\lambda'_{21}, \lambda'_{11}, \lambda'_{22}}\,.
% \end{align}
% }
Finally, combining \eqref{eq:sign1}, \eqref{eq:sign2}, one can write
(up to a global undetermined sign)
\begin{align}\label{}
 \sigma_{\lambda'_{21}, \lambda'_{11}, \lambda'_{22}}=(-1)^{\lambda'_{22}-\lambda_{31}}\,.
\end{align}
\endproof
Let us emphasize that the undetermined global sign does not play a role in the upcoming
computations since only the product of $T$ and $T^{-1}$ is involved.

For particular representations, the previous result simplifies.
\begin{coro}\label{coro:symmetric_racah}
For the symmetric representation (\textit{i.e.} $\lambda_{32}=\lambda_{33}=\lambda_{22}$,
$\lambda_{31}=-2\lambda_{32}$), $\tau$ is given (up to a global sign) by
\begin{align}\label{eq:1row}
 \tau_{\Lambda'\Lambda}=&
 \delta_{\lambda'_{11}, \lambda_{11}}
 \delta_{\lambda_{21}+2\lambda_{32}, \lambda'_{11}- \lambda'_{21}}
 \delta_{\lambda'_{22},\lambda_{32}} \delta_{\lambda_{22},\lambda_{32}} (-1)^{\lambda_{32}-\lambda_{21}}
 \frac{(\lambda_{21}-\lambda_{32})!}{(\lambda_{11}-\lambda_{21}-3\lambda_{32})!}\,.
\end{align}
For the representation characterized by $\lambda_{31}=\lambda_{32}=\lambda_{21}$,
$\lambda_{33}=-2\lambda_{31}$, one gets
  \begin{align}\label{eq:2row}
  \tau_{\Lambda'\Lambda}=&
  \delta_{\lambda'_{11}, \lambda_{11}}
  \delta_{\lambda_{11}-\lambda_{22}-2\lambda_{31},  \lambda'_{22}}
  \delta_{\lambda'_{21},\lambda_{31}} \delta_{\lambda_{21},\lambda_{31}}(-1)^{\lambda'_{22}-\lambda_{31}}
  \frac{(\lambda_{22}+2\lambda_{31})!}{(\lambda_{11}-\lambda_{22})!}\,.
 \end{align}
 \end{coro}
\proof For the symmetric representation, the Racah polynomial in \eqref{eq:tau} reduces to
\begin{equation}\label{eq:int2}
 R_{-2\lambda_{32}-\lambda_{21}}(X)= {_{3}F}_{2}\left( \genfrac{}{}{0pt}{0}{2\lambda_{32}+\lambda_{21}\, ,\ 3\lambda_{32}-1\,,\ \lambda_{11}-\lambda_{21}}{\lambda_{11}+2\lambda_{32}\, ,\ 3\lambda_{32} }\Big| 1\right)\;.
\end{equation}
Using Thomae's transformation formula
(see \cite[relation~(3.1.1)]{GasperRahman2004} for example)
\begin{equation}
  {_{3}F}_{2}\left( \genfrac{}{}{0pt}{0}{-n\, ,\ a\,,\ b}{c\,,\ d}\Big| 1\right)=\frac{(d-b)_n}{(d)_n} {_{3}F}_{2}\left( \genfrac{}{}{0pt}{0}{-n\, ,\ c-a\,,\ b}{c\,,\ 1+b-d-n}\Big| 1\right)\,,
\end{equation}
relation \eqref{eq:int2} becomes
\begin{align}
 R_{-2\lambda_{32}-\lambda_{21}}(X)&=\frac{(3\lambda_{32}-\lambda_{11}+\lambda_{21})_{-\lambda_{21}-2\lambda_{32}}}{(3\lambda_{32})_{-\lambda_{21}-2\lambda_{32}} }
 {_{2}F}_{1}\left( \genfrac{}{}{0pt}{0}{2\lambda_{32}+\lambda_{21}\, ,\ \lambda_{11}-\lambda_{21}}{\lambda_{11}+2\lambda_{32} }\Big| 1\right)\;,\\
 &=\frac{(3\lambda_{32}-\lambda_{11}+\lambda_{21})_{-\lambda_{21}-2\lambda_{32}}  (2\lambda_{32}+\lambda_{21})_{-\lambda_{21}-2\lambda_{32}}  }
 {(3\lambda_{32})_{-\lambda_{21}-2\lambda_{32}} (\lambda_{11}+2\lambda_{32})_{-\lambda_{21}-2\lambda_{32}}}
\,.
\end{align}
This proves relation \eqref{eq:1row}. Relation \eqref{eq:2row} is proven by remarking that $R_0(X)=1$.
\endproof

\section{General matrix elements for $SO(3)$ rotations in a general $\sl_3$ irrep
and a new family of bispectral trivariate hybrid functions}
\label{sec:Rgeneral}

\begin{theo} \label{pr:res}
The matrix elements of the change of basis
$\sigma$ induced by the rotation $S$ given by \eqref{eq:S} are, for $\Lambda,\Lambda'$ GT patterns,
 \begin{align}
  &\sigma_{\Lambda',\Lambda}=\sum_{n=\text{max}\{-\lambda_{21} , -\lambda'_{21}\} }^{\text{min}\{-\lambda_{22} , -\lambda'_{22}\}}
  \sum_{\ell=\ell_{min} }^{\text{min}\{\lambda_{31},n-\lambda_{33}\}  }
  \mu_{\Lambda',\Lambda}(\ell,n)
  K_{n+\lambda'_{21}}(\lambda'_{11}-\lambda'_{22};\sin^2(\phi),\lambda'_{21}-\lambda'_{22})\nonumber\\
 \times &\widetilde{R}_{\lambda_{31}-\lambda'_{21}}(\lambda_{31}-\ell;\lambda_{32}-\lambda_{31}-1,\lambda'_{21}+\lambda'_{22}+\lambda_{33}-1,n+\lambda'_{21}+\lambda'_{22}-\lambda_{31}-1,-\lambda'_{21}-\lambda'_{22}-\lambda_{31}-1)\nonumber\\
 \times  & K_{\ell+\lambda_{21}+\lambda_{22}}(\, \ell+\lambda'_{21}+\lambda'_{22} \,;\,\sin^2(\theta),2\ell-n)\nonumber\\
  \times &\widetilde R_{\lambda_{31}-\lambda_{21}}( \lambda_{31}-\ell;\lambda_{32}-\lambda_{31}-1,\lambda_{21}+\lambda_{22}+\lambda_{33}-1,n +\lambda_{21}+\lambda_{22}-\lambda_{31}-1,-\lambda_{21}-\lambda_{22}-\lambda_{31}-1)\nonumber\\
 \times &  K_{\lambda_{11}-\lambda_{22}}(n+\lambda_{21};  \sin^2(\chi),\lambda_{21}-\lambda_{22})\,,
 \label{eq:new_polynomials}
 \end{align}
 % \begin{align}
 %  &\sigma_{\Lambda,\Lambda'}=\sum_{n=\text{max}\{-\lambda_{21} , -\lambda'_{21}\} }^{\text{min}\{-\lambda_{22} , -\lambda'_{22}\}}
 %  \sum_{\ell=\ell_{min} }^{\text{min}\{\lambda_{31},n-\lambda_{33}\}  }
 %  \mu_{\Lambda,\Lambda'}(\ell,n)
 %  K_{n+\lambda_{21}}(\lambda_{11}-\lambda_{22};\sin^2(\phi),\lambda_{21}-\lambda_{22})\nonumber\\
 % \times &\widetilde{R}_{\lambda_{31}-\lambda_{21}}(\lambda_{31}-\ell;\lambda_{32}-\lambda_{31}-1,\lambda_{21}+\lambda_{22}+\lambda_{33}-1,n+\lambda_{21}+\lambda_{22}-\lambda_{31}-1,-\lambda_{21}-\lambda_{22}-\lambda_{31}-1)\nonumber\\
 % \times  & K_{\ell+\lambda'_{21}+\lambda'_{22}}(\, \ell+\lambda_{21}+\lambda_{22} \,;\,\sin^2(\theta),2\ell-n)\nonumber\\
 %  \times &\widetilde R_{\lambda_{31}-\lambda'_{21}}( \lambda_{31}-\ell;\lambda_{32}-\lambda_{31}-1,\lambda'_{21}+\lambda'_{22}+\lambda_{33}-1,n +\lambda'_{21}+\lambda'_{22}-\lambda_{31}-1,-\lambda'_{21}-\lambda'_{22}-\lambda_{31}-1)\nonumber\\
 % \times &  K_{\lambda'_{11}-\lambda'_{22}}(n+\lambda'_{21};  \sin^2(\chi),\lambda'_{21}-\lambda'_{22})\,,
 % \label{eq:new_polynomials}
 % \end{align}
where
$\ell_{min}=\text{max}\{\lambda_{32},n-\lambda_{32},-\lambda_{21}-\lambda_{22},n+\lambda_{21}+\lambda_{22},-\lambda'_{21}-\lambda'_{22},n+\lambda'_{21}+\lambda'_{22}\}$,
 \begin{align}
 \mu_{\Lambda',\Lambda}(\ell,n)&=(-1)^{\lambda'_{11}+\ell-n} t(\lambda'_{21},n+\lambda'_{21}+\lambda'_{22},\lambda'_{22})
  t(\ell,n+\lambda_{21}+\lambda_{22},n-\ell)\\
  \times &\frac{
  (\lambda'_{21}-\lambda'_{22})!\, (2\ell-n)!\, (\lambda_{21}-\lambda_{22})!\ \cos^{\lambda'_{21}-\lambda'_{22}}(\phi)\cos^{2\ell-n}(\theta) \cos^{\lambda_{21}-\lambda_{22}}(\chi)}
 {(n+\lambda'_{21})!(\lambda'_{21} -\lambda'_{11})!( \ell+\lambda_{21}+\lambda_{22})!
  (\ell-n-\lambda'_{21}-\lambda'_{22})!  (\lambda_{11}-\lambda_{22})!(-n-\lambda_{22})!}
  \nonumber \\
\times&\tan^{n+\lambda'_{11}+\lambda'_{21}-\lambda'_{22}}(\phi)
 \tan^{2\ell+\lambda'_{21}+\lambda'_{22}+\lambda_{21}+\lambda_{22}}(\theta)
 \tan^{n+\lambda_{11}+\lambda_{21}-\lambda_{22}}(\chi)\,, \nonumber
 \end{align}
 % \begin{align}
 % \mu_{\Lambda,\Lambda'}(\ell,n)&=(-1)^{\lambda_{11}+\ell-n} t(\lambda_{21},n+\lambda_{21}+\lambda_{22},\lambda_{22})
 %  t(\ell,n+\lambda'_{21}+\lambda'_{22},n-\ell)\\
 %  \times &\frac{
 %  (\lambda_{21}-\lambda_{22})!\, (2\ell-n)!\, (\lambda'_{21}-\lambda'_{22})!\ \cos^{\lambda_{21}-\lambda_{22}}(\phi)\cos^{2\ell-n}(\theta) \cos^{\lambda'_{21}-\lambda'_{22}}(\chi)}
 % {(n+\lambda_{21})!(\lambda_{21} -\lambda_{11})!( \ell+\lambda'_{21}+\lambda'_{22})!
 %  (\ell-n-\lambda_{21}-\lambda_{22})!  (\lambda'_{11}-\lambda'_{22})!(-n-\lambda'_{22})!}
 %  \nonumber \\
% \times&\tan^{n+\lambda_{11}+\lambda_{21}-\lambda_{22}}(\phi)
 % \tan^{2\ell+\lambda_{21}+\lambda_{22}+\lambda'_{21}+\lambda'_{22}}(\theta)
 % \tan^{n+\lambda'_{11}+\lambda'_{21}-\lambda'_{22}}(\chi)\,, \nonumber
 % \end{align}
 and $t(x,y,z)$ is defined by \eqref{eq:tbis}.
\end{theo}

\proof From the discussion at the end of Section \ref{sec:usl3_gt}, one gets
\begin{align}\label{eq:int1}
 \sigma_{\Lambda',\Lambda} =\sum_{\Lambda^{(1)},\Lambda^{(2)},\Lambda^{(3)},\Lambda^{(4)}\in \mathcal{P}_\lambda}
 (\rho_\phi)_{\Lambda',\Lambda^{(1)}}\, (\tau^{-1})_{\Lambda^{(1)},\Lambda^{(2)}}\, (\rho_{\theta})_{\Lambda^{(2)},\Lambda^{(3)}}\, (\tau)_{\Lambda^{(3)},\Lambda^{(4)}}\, (\rho_\chi)_{\Lambda^{(4)},\Lambda}
 % \sigma_{\Lambda,\Lambda'} =\sum_{\Lambda^{(1)},\Lambda^{(2)},\Lambda^{(3)},\Lambda^{(4)}\in \mathcal{P}_\lambda}
 % (\rho_\phi)_{\Lambda,\Lambda^{(1)}}\, (\tau^{-1})_{\Lambda^{(1)},\Lambda^{(2)}}\, (\rho_{\theta})_{\Lambda^{(2)},\Lambda^{(3)}}\, (\tau)_{\Lambda^{(3)},\Lambda^{(4)}}\, (\rho_\chi)_{\Lambda^{(4)},\Lambda'}
\end{align}
The components of $\rho_\phi$, $\rho_\theta$, $\rho_\chi$ are given by \eqref{eq:rho} and
the ones of $\tau$ by \eqref{eq:tau}.
The ones of $\tau^{-1}$ are computed as follows.
Using the orthogonality relation \eqref{eq:ortho}, one gets
\begin{equation}
 (\tau^{-1})_{\Lambda,\Lambda'}=\left(\frac{N_{\Lambda'}}{N_\Lambda}\right)^{2} \tau_{\Lambda',\Lambda}
\end{equation}
which leads to
\begin{align}\label{eq:ttaui}
 (\tau^{-1})_{\Lambda,\Lambda'}&= \delta_{\lambda'_{11}, \lambda_{11}}\delta_{\lambda_{21}+\lambda_{22}, \lambda'_{11}- \lambda'_{21}-\lambda'_{22}}
 t(\lambda_{21},\lambda_{11},\lambda_{22}) (-1)^{\lambda'_{22}-\lambda_{21}} \\
 \times  &
  \widetilde{R}_{\lambda_{31}-\lambda_{21}}(\lambda_{31}-\lambda'_{21} ;  \lambda_{32}-\lambda_{31}-1 ,\lambda_{21}+\lambda_{22} +\lambda_{33}-1 ,\lambda_{11}-\lambda_{31}-1 ,-\lambda_{21}-\lambda_{22} -\lambda_{31}-1)\,.\nonumber
\end{align}
The expression \eqref{eq:int1} can be expressed now as follows
\begin{align}
 \sigma_{\Lambda',\Lambda} =\sum_{\lambda_{21}^{(1)},\lambda_{11}^{(1)},\lambda_{22}^{(1)},\lambda_{21}^{(2)},\dots,\lambda_{22}^{(4)}}&
 \delta_{\lambda'_{21},\lambda_{21}^{(1)}} \delta_{\lambda'_{22},\lambda_{22}^{(1)}} \overline{\rho}_{\Lambda',\Lambda^{(1)}}\
 \delta_{\lambda_{11}^{(1)},\lambda_{11}^{(2)}} \delta_{\lambda_{21}^{(1)}+\lambda_{22}^{(1)},\lambda_{11}^{(2)}-\lambda_{21}^{(2)}-\lambda_{22}^{(2)}}  \widetilde{\tau}_{\Lambda^{(1)},\Lambda^{(2)}}\nonumber \\
\times& \delta_{\lambda_{21}^{(2)},\lambda_{21}^{(3)}} \delta_{\lambda_{22}^{(2)},\lambda_{22}^{(3)}}  \overline{\rho}_{\Lambda^{(2)},\Lambda^{(3)}}\
  \delta_{\lambda_{11}^{(3)},\lambda_{11}^{(4)}}  \delta_{\lambda_{21}^{(3)}+\lambda_{22}^{(3)},\lambda_{11}^{(4)}-\lambda_{21}^{(4)}-\lambda_{22}^{(4)}} \overline{\tau}_{\Lambda^{(3)},\Lambda^{(4)}}\nonumber \\
\times&\delta_{\lambda_{21}^{(4)},\lambda_{21}} \delta_{\lambda_{22}^{(4)},\lambda_{22}} \overline{\rho}_{\Lambda^{(4)},\Lambda}\,,
\end{align}
where  $\overline{\rho}_{\Lambda,\Lambda'}$ (resp. $\overline{\tau}_{\Lambda,\Lambda'}$,  $\widetilde{\tau}_{\Lambda,\Lambda'}$) are the expressions following the $\delta$'s in \eqref{eq:rho} (resp. \eqref{eq:tau}, \eqref{eq:ttaui}).
Combining everything together and setting $n=\lambda_{11}^{(1)}-\lambda'_{21}-\lambda'_{22}$, $\ell=\lambda_{21}^{(2)}$, the theorem is proven.
% \begin{align}
%  \sigma_{\Lambda,\Lambda'} =\sum_{\lambda_{21}^{(1)},\lambda_{11}^{(1)},\lambda_{22}^{(1)},\lambda_{21}^{(2)},\dots,\lambda_{22}^{(4)}}&
%  \delta_{\lambda_{21},\lambda_{21}^{(1)}} \delta_{\lambda_{22},\lambda_{22}^{(1)}} \overline{\rho}_{\Lambda,\Lambda^{(1)}}\
%  \delta_{\lambda_{11}^{(1)},\lambda_{11}^{(2)}} \delta_{\lambda_{21}^{(1)}+\lambda_{22}^{(1)},\lambda_{11}^{(2)}-\lambda_{21}^{(2)}-\lambda_{22}^{(2)}}  \widetilde{\tau}_{\Lambda^{(1)},\Lambda^{(2)}}\nonumber \\
% \times& \delta_{\lambda_{21}^{(2)},\lambda_{21}^{(3)}} \delta_{\lambda_{22}^{(2)},\lambda_{22}^{(3)}}  \overline{\rho}_{\Lambda^{(2)},\Lambda^{(3)}}\
%   \delta_{\lambda_{11}^{(3)},\lambda_{11}^{(4)}}  \delta_{\lambda_{21}^{(3)}+\lambda_{22}^{(3)},\lambda_{11}^{(4)}-\lambda_{21}^{(4)}-\lambda_{22}^{(4)}} \overline{\tau}_{\Lambda^{(3)},\Lambda^{(4)}}\nonumber \\
% \times&\delta_{\lambda_{21}^{(4)},\lambda'_{21}} \delta_{\lambda_{22}^{(4)},\lambda'_{22}} \overline{\rho}_{\Lambda^{(4)},\Lambda'}\,,
% \end{align}
% where  $\overline{\rho}_{\Lambda,\Lambda'}$ (resp. $\overline{\tau}_{\Lambda,\Lambda'}$,  $\widetilde{\tau}_{\Lambda,\Lambda'}$) are the expressions following the $\delta$'s in \eqref{eq:rho} (resp. \eqref{eq:tau}, \eqref{eq:ttaui}).
% Combining everything together and setting $n=\lambda_{11}^{(1)}-\lambda_{21}-\lambda_{22}$, $\ell=\lambda_{21}^{(2)}$, the theorem is proven.
\endproof

The element $\sigma_{\Lambda',\Lambda}$ given in \eqref{eq:new_polynomials} can be seen as a function of the three variables $\Lambda'=(\lambda_{21}',\lambda_{11}',\lambda_{22}')$.
In the following, we shall give a difference relation satisfied by this function. We shall also provide a relation between the functions for different values of $\Lambda=(\lambda_{21},\lambda_{11},\lambda_{22})$, called recurrence relation,
in reference to the similar relations satisfied by the Krawtchouk or Racah polynomials.

\paragraph{Recurrence relations.}  To simplify the notations, we denote by $S_{ij}$ the entries of the rotation $S$ given by \eqref{eq:S}.
Relation \eqref{eq:fond1} for this rotation reads as follows, for any $g \in sl_3$,
\begin{equation}\label{eq:fonS}
 \sigma\cdot \Psi_S(g)  \cdot  \xi_\Lambda =  g  \cdot  \sigma  \cdot  \xi_\Lambda\,,
\end{equation}
and provides three different recurrence relations for $\sigma_{\Lambda',\Lambda}$ when one
chooses either $g=H,Y,J$.

Explicitly, for $g=H$, it reads:
\begin{align}
&\Big( (S_{11}S_{11}-S_{12}S_{12})\lambda_{11}+(S_{21}S_{21}-S_{22}S_{22})(\lambda_{21}+\lambda_{22}-\lambda_{11})-(S_{31}S_{31}-S_{32}S_{32})(\lambda_{21}+\lambda_{22}) \Big) \sigma_{\Lambda',\Lambda}\nonumber\\
+&(S_{11}S_{21}-S_{12}S_{22})(\lambda_{21}-\lambda_{11})(\lambda_{11}-\lambda_{22}+1)\sigma_{\Lambda',\Lambda+\delta^{11}}
+(S_{21}S_{11}-S_{22}S_{12})\sigma_{\Lambda',\Lambda-\delta^{11}}  \nonumber\\
+&(S_{21}S_{31}-S_{22}S_{32})\left( \frac{\lambda_{31}-\lambda_{21}   }{\lambda_{21}-\lambda_{22}+1}  \sigma_{\Lambda',\Lambda+\delta^{21}}
  +\frac{ \lambda_{31}-\lambda_{22}+1}{\lambda_{21}-\lambda_{22}+1} \sigma_{\Lambda',\Lambda+\delta^{22}}    \right) \nonumber\\
+&(S_{31}S_{21}-S_{32}S_{22})\left(\frac{(\lambda_{21}-\lambda_{32})(\lambda_{21}-\lambda_{33}+1)(\lambda_{21}-\lambda_{11})   }{\lambda_{21}-\lambda_{22}+1} \sigma_{\Lambda',\Lambda-\delta^{21}}\right. \nonumber\\
 &\hspace{3cm}\left. +\frac{(\lambda_{11}-\lambda_{22}+1)(\lambda_{32}-\lambda_{22} +1)(\lambda_{22}-\lambda_{33})  }{\lambda_{21}-\lambda_{22}+1} \sigma_{\Lambda',\Lambda-\delta^{22}}    \right) \nonumber\\
+&(S_{11}S_{31}-S_{12}S_{32})\left( \frac{(\lambda_{11}-\lambda_{22}+1)(\lambda_{31}-\lambda_{21})   }{\lambda_{21}-\lambda_{22}+1}   \sigma_{\Lambda',\Lambda+\delta^{11}+\delta^{21}}
 -\frac{(\lambda_{21}-\lambda_{11})( \lambda_{31}-\lambda_{22}+1)}{\lambda_{21}-\lambda_{22}+1}  \sigma_{\Lambda',\Lambda+\delta^{11}+\delta^{22}}    \right) \nonumber\\
 +&(S_{31}S_{11}-S_{32}S_{12})\left( \frac{(\lambda_{21}-\lambda_{32})(\lambda_{21}-\lambda_{33}+1)  }{\lambda_{21}-\lambda_{22}+1}  \sigma_{\Lambda',\Lambda-\delta^{11}-\delta^{21}}
  -\frac{(\lambda_{32}-\lambda_{22} +1)(\lambda_{22}-\lambda_{33})  }{\lambda_{21}-\lambda_{22}+1} \sigma_{\Lambda',\Lambda-\delta^{11}-\delta^{22}}    \right) \nonumber\\
=&(2\lambda'_{11} -\lambda'_{21} -\lambda'_{22} )\sigma_{\Lambda',\Lambda}\,.
\end{align}
For $g=Y,J$, the recurrence relations can be computed similarly.
The `hybrid' property of the functions $\sigma_{\Lambda',\Lambda}$ shows up in particular
in the fact that the eigenvalue in the R.H.S. is also linear in $\Lambda'$ for $g=Y$
while it is quadratic for $g=J$.

\paragraph{Difference relations.}
To get the difference relations, replace $g$ by $\Psi_{S^t}(g)$ in relation
\eqref{eq:fonS} to find
\begin{equation}\label{eq:fonS2}
 \sigma\cdot g  \cdot  \xi_\Lambda =  \Psi_{S^t}(g)  \cdot  \sigma  \cdot  \xi_\Lambda\,,
\end{equation}
where $S^t$ stands for the transposition of $S$ and we have used that $\Psi_{S}(\Psi_{S^t}(g)) = g$.
Using this relation for $g=H,Y,J$, one gets three difference relations.
Namely, for $g=H$, it becomes
\begin{align}
&(2\lambda_{11} -\lambda_{21} -\lambda_{22} )\sigma_{\Lambda',\Lambda}=\\
&\Big( (S_{11}S_{11}-S_{21}S_{21})\lambda'_{11}+(S_{12}S_{12}-S_{22}S_{22})(\lambda'_{21}+\lambda'_{22}-\lambda'_{11})-(S_{13}S_{13}-S_{23}S_{23})(\lambda'_{21}+\lambda'_{22}) \Big) \sigma_{\Lambda',\Lambda}\nonumber\\
+&(S_{11}S_{12}-S_{21}S_{22})(\lambda'_{21}-\lambda'_{11}+1)(\lambda'_{11}-\lambda'_{22})\sigma_{\Lambda'-\delta^{11},\Lambda}
+(S_{12}S_{11}-S_{22}S_{21})\sigma_{\Lambda'+\delta^{11},\Lambda}  \nonumber\\
+&(S_{12}S_{13}-S_{22}S_{23})\left( \frac{\lambda'_{31}-\lambda'_{21}  +1 }{\lambda'_{21}-\lambda'_{22}}  \sigma_{\Lambda'-\delta^{21},\Lambda}
  +\frac{ \lambda'_{31}-\lambda'_{22}+2}{\lambda'_{21}-\lambda'_{22}+2} \sigma_{\Lambda'-\delta^{22},\Lambda}    \right) \nonumber\\
+&(S_{13}S_{12}-S_{23}S_{22})\left(\frac{(\lambda'_{21}-\lambda'_{32}+1)(\lambda'_{21}-\lambda'_{33}+2)(\lambda'_{21}-\lambda'_{11}+1)   }{\lambda'_{21}-\lambda'_{22}+2} \sigma_{\Lambda'+\delta^{21},\Lambda}\right. \nonumber\\
 &\hspace{3cm}\left. +\frac{(\lambda'_{11}-\lambda'_{22})(\lambda'_{32}-\lambda'_{22} )(\lambda'_{22}-\lambda'_{33}+1)  }{\lambda'_{21}-\lambda'_{22}} \sigma_{\Lambda'+\delta^{22},\Lambda}    \right) \nonumber\\
  +&(S_{11}S_{13}-S_{21}S_{23})\left( \frac{(\lambda'_{11}-\lambda'_{22})(\lambda'_{31}-\lambda'_{21}+1)   }{\lambda'_{21}-\lambda'_{22}}   \sigma_{\Lambda'-\delta^{11}-\delta^{21},\Lambda}
  -\frac{(\lambda'_{21}-\lambda'_{11}+1)( \lambda'_{31}-\lambda'_{22}+2)}{\lambda'_{21}-\lambda'_{22}+2}  \sigma_{\Lambda'-\delta^{11}-\delta^{22},\Lambda}    \right) \nonumber\\
  +&(S_{13}S_{11}-S_{23}S_{21})\left( \frac{(\lambda'_{21}-\lambda'_{32}+1)(\lambda'_{21}-\lambda'_{33}+2)  }{\lambda'_{21}-\lambda'_{22}+2}  \sigma_{\Lambda'+\delta^{11}+\delta^{21},\Lambda}
  -\frac{(\lambda'_{32}-\lambda'_{22} )(\lambda'_{22}-\lambda'_{33}+1)  }{\lambda'_{21}-\lambda'_{22}} \sigma_{\Lambda'+\delta^{11}+\delta^{22},\Lambda}    \right)\,. \nonumber
\end{align}
For $g=Y,J$, similar relations are obtained.
Again, the `hybrid' property of the functions $\sigma_{\Lambda',\Lambda}$
is illustrated by the fact that the eigenvalue in the L.H.S. is also linear in $\Lambda$
for $g=Y$ while it is quadratic for $g=J$.

These sets of recurrence relations and difference relations establish the bispectrality
of the functions introduced in \eqref{eq:new_polynomials}.

\section{Particular examples}
\label{sec:examples}

\subsection{Bivariate Krawtchouk polynomials of Griffiths type}
Let us consider the case of symmetric representations of $\sl_3$ which has also been
looked at in Corollary \ref{coro:symmetric_racah}. We recall that the GT pattern of the symmetric representations are such that
$\lambda_{32}=\lambda_{33}$, $\lambda_{31}=-2\lambda_{33}$ which implies that $\lambda_{22} =\lambda_{32}=\lambda_{33}$.

\begin{theo} \label{theo:symmetric}
The matrix elements of the change of basis
$\sigma$ induced by the rotation $S$ given by \eqref{eq:S} are, for $\Lambda,\Lambda'$ GT patterns of symmetric
representations,
\begin{align}\label{eq:krawtchouk_griffiths}
\begin{aligned}
 \sigma_{\Lambda,\Lambda'}
 &=\sum_{\ell=max\{0,\lambda_{21}-\lambda'_{21}\}}^{\lambda_{21}-\lambda_{33}}
 f_\ell(\lambda_{21},\lambda_{11},\lambda'_{21},\lambda'_{11})~
 K_{\ell}(\lambda_{11}-\lambda_{33};\sin^{2}\phi,\lambda_{21}-\lambda_{33})\\
 &\qquad\times K_{\ell+\lambda'_{21}-\lambda_{21}}(\ell;
                         \sin^{2}\theta,\ell-\lambda_{21}-2\lambda_{33})~
 K_{\lambda'_{11}-\lambda_{33}}(\ell+\lambda'_{21}-\lambda_{21},
                         \sin^{2}\chi,\lambda'_{21}-\lambda_{33})\,,
\end{aligned}
\end{align}
with
\begin{align}\label{}
\begin{aligned}
 f_\ell(\lambda_{21},\lambda_{11},\lambda'_{21},\lambda'_{11})&=
 (-1)^{\lambda_{11}-2\lambda_{21}+2\lambda_{33}-\lambda'_{22}}
(\tan\phi)^{\ell+\lambda_{11}-\lambda_{33}}
(\tan\theta)^{2\ell+\lambda'_{21}-\lambda_{21}}
(\tan\chi)^{\ell+\lambda'_{21}-\lambda_{21}+\lambda'_{11}-\lambda_{33}}\\
 &\qquad\times
\frac{
(\cos\phi)^{\lambda_{21}-\lambda_{33}}
(\cos\theta)^{\ell-\lambda_{21}-2\lambda_{33}}
(\cos\chi)^{\lambda'_{21}-\lambda_{33}}
}
{\ell!(\ell+\lambda'_{21}-\lambda_{21})!(\ell-\lambda_{21}-2\lambda_{33})!
 (-\ell+\lambda_{21}-\lambda_{33})!}
\\
 &\qquad\times
\frac{[(\lambda_{21}-\lambda_{33})!(\lambda'_{21}-\lambda_{33})!]^{2}}
{(\lambda_{21}-\lambda_{11})!(\lambda'_{11}-\lambda_{33})!(-2\lambda_{33}-\lambda_{21})}
\,.
\end{aligned}
\end{align}
\end{theo}
\proof
The calculation proceeds similarly to what was done in Theorem \ref{pr:res}, this
time using the expressions for the symmetric representation in Corollary
\ref{coro:symmetric_racah}. Combining everything together and setting
$\ell=\lambda_{11}^{(1)}-\lambda_{33}$, the result follows.
\endproof
Making use of the symmetry property of Krawtchouk polynomials
\begin{align}\label{}
 K_{n}(x;p,N)=\frac{n!}{(N-n)!}(-1)^{x+n}p^{n-x}(1-p)^{x+n-N}
  K_{N-n}(N-x;p,N),
\end{align}
one can compare with the expressions found in
\cite[Equations~(10.4)--(10.5)]{GenestVinetetal2013b} for the bivariate
Griffiths--Krawtchouk polynomials.

\subsection{Bivariate hybrid functions in terms of Krawtchouk and Racah polynomials}

We consider the transformation $T$ followed by a rotation around the $z$-axis:
\begin{align}S= R_{\eta}^zT =\begin{pmatrix}
                                  \cos(\eta) & 0 & \sin(\eta)\\
                                 - \sin(\eta)&0&  \cos(\eta) \\
                                 0&-1 & 0
                                 \end{pmatrix}\ .
                                 \end{align}
It can be obtained with the following choice of angles in the parametrization of a general
rotation used in the preceding sections:
\begin{align}\theta=\frac{\pi}{2}\,,\ \ \ \ \ \phi=
 -\frac{\pi}{2}\,,\ \ \ \ \ \ \chi=\eta+\frac{\pi}{2}\ .\end{align}
It is simpler to calculate the matrix elements of $S$ directly rather than putting these
parameters in the general formula. Indeed, we have:
\begin{align} \sigma_{\Lambda',\Lambda}=\sum_{\Lambda^{(1)}}\tau_{\Lambda',\Lambda^{(1)}}(\rho_{\eta})_{\Lambda^{(1)},\Lambda}\ .\end{align}
Using the explicit expressions found for the matrix elements $\tau$ and $\rho_{\eta}$, we
get the following result.
\begin{prop} For $\Lambda,\Lambda'$ GT patterns, the matrix elements of the change of basis $\sigma$ corresponding to the rotation $S$ are given (up to a
global undetermined sign) by  ,
%\begin{align}
%  \sigma_{\Lambda'\Lambda}= & 0 & \qquad\text{if $(\lambda_{21},\lambda'_{11},\lambda_{22})$ is not a GT pattern\,,}
% \end{align}
%and if $(\lambda_{21},\lambda'_{11},\lambda_{22})$ is a GT pattern:
\begin{align}
 \sigma_{\Lambda'\Lambda}
 =&\delta_{\lambda_{21}+\lambda_{22}, \lambda'_{11}- \lambda'_{21}-\lambda'_{22}}\,
  t(\lambda'_{21},\lambda'_{11},\lambda'_{22})
 (-1)^{2\lambda_{21}+\lambda'_{21}}
 \frac{(\lambda_{21}-\lambda_{22})!\tan^{\lambda'_{11}+\lambda_{11}-2\lambda_{22}}(\eta)
      \cos^{\lambda_{21}-\lambda_{22}}(\eta)}
     {(\lambda_{11}-\lambda_{22})!(\lambda_{21}-\lambda'_{11})!} \nonumber \\
  \times & K_{\lambda_{11}-\lambda_{22}}(\lambda'_{11}-\lambda_{22};
              \sin^2(\eta),\lambda_{21}-\lambda_{22})
  ~\widetilde{R}_{\lambda_{31}-\lambda_{21}}(\lambda_{31}-\lambda'_{21}  ;
              \alpha ,\beta ,\gamma ,\delta)\,,
\end{align}
with
\begin{align}
\begin{aligned}
 &\alpha= \lambda_{32}-\lambda_{31}-1\,,\qquad\quad \beta=\lambda'_{11}-\lambda'_{21}-\lambda'_{22} +\lambda_{33}-1\,,\\
  &\gamma= \lambda'_{11}-\lambda_{31}-1\,,\qquad\quad \delta= \lambda'_{21}+\lambda'_{22}-\lambda'_{11} -\lambda_{31}-1\,.
\end{aligned}
\end{align}
\end{prop}

In the previous proposition, a product of Krawtchouk and Racah polynomials appears in each sector where $\lambda_{21}+\lambda_{22}= \lambda'_{11}- \lambda'_{21}-\lambda'_{22}$.
It naturally suggests to consider the following bivariate function:
\begin{align}\label{hybrid_KR}
 P_{n_1,n_2}(x_1,x_2)
 =& K_{n_1}(x_1-n_2;
              \sin^2(\eta),N-2n_2)
  ~\widetilde{R}_{n_2}(x_2  ;
              \alpha ,\beta,x_1-N-1 ,\delta)\,,
\end{align}
where $N$ is a positive integer and $\alpha,\beta,\delta,\eta$ are parameters. In the
formula of the proposition, $N$ corresponds to $2\lambda_{31}-\lambda_{21}-\lambda_{22}$ and
 $x_1$, $x_2$, $n_1$ and $n_2$  to
 \begin{subequations}
\begin{align}
 & x_1=\lambda'_{22}+\lambda'_{21}+\lambda_{31}\,, && x_2=\lambda_{31}-\lambda'_{21}\,,\\
 &n_1=\lambda_{11}-\lambda_{22}\,,&& n_2 =\lambda_{31}-\lambda_{21}\,.
\end{align}
\end{subequations}
For these choices, the L.H.S. of the recurrence relations of $P_{n_1,n_2}(x_1,x_2)$ obtained by \eqref{eq:fonS} involves only terms of the form $P_{n_1',n_2'}(x_1,x_2)$.
More precisely, \eqref{eq:fonS} for $g=e_{11}$ provides a recurrence relation involving only $ P_{n_1+\varepsilon ,n_2}(x_1,x_2)$ and \eqref{eq:fonS} for $g=e_{21}e_{12}$
provides a recurrence relation involving only $ P_{n_1+\varepsilon,n_2}(x_1,x_2)$, $ P_{n_1,n_2+\varepsilon}(x_1,x_2)$, $ P_{n_1+\varepsilon,n_2-\varepsilon}(x_1,x_2)$ and $ P_{n_1+2\varepsilon,n_2-\varepsilon}(x_1,x_2)$, with $\varepsilon\in\{-1,0,1\}$.
Similarly, the R.H.S. of the difference relations of $ P_{n_1,n_2}(x_1,x_2)$ obtained by \eqref{eq:fonS2} involves only $ P_{n_1,n_2}(x_1',x_2')$:
for $g=e_{11}$, one obtains   $ P_{n_1,n_2}(x_1+\varepsilon,x_2)$ and $ P_{n_1,n_2}(x_1+\varepsilon,x_2-\varepsilon)$
and for $g=J$, one obtains  $P_{n_1,n_2}(x_1,x_2+\varepsilon)$, with $\varepsilon\in\{-1,0,1\}$.

Formula \eqref{hybrid_KR} resembles very much the Tratnik construction of bivariate orthogonal
polynomials, except that here it is an hybrid of two different univariate polynomials.
Such hybrid constructions already appeared, see for example \cite{Dunkl1976,
CrampeVinetetal2023} for a formula like \eqref{hybrid_KR} using dual Hahn polynomials
instead of Racah polynomials.

\section{The Racah algebra in $U(sl_3)$ and the centralizer of the Cartan subalgebra}
\label{sec:Racah-algebra}
This section contains some algebraic results concerning the elements $J$ and
$\Psi_T(J)$ introduced previously in Section \ref{sec:T_racah}.

\paragraph{Racah algebra.}
The observation that the change of basis associated to $T$ is related to the Racah
polynomials can be understood algebraically by remarking that $J$ and $\overline
J=\Psi_T(J)$ satisfy the relations of the Racah algebra \eqref{eq:racah3}.
Indeed, let us define
\begin{subequations}\label{Racah-sl3}
\begin{equation}
 K:= \left[\,J,\overline J\, \right]= e_{31}e_{12}e_{23}-e_{32}e_{21}e_{13}\,.
\end{equation}
Then, by a lengthy but straightforward calculation using the defining relations of $U(sl_3)$, one can show that
\begin{align}
&[J,K]=2J^2+2\{J,\overline{J}\}-a J +b^+\,,\\
&[K,\overline{J}]=2\overline{J}^2+2\{J,\overline{J}\} -a\overline{J} +b^-\,,
\end{align}
\end{subequations}
where
\begin{subequations}\label{eq:realization_racah_central_elements}
\begin{align}
 &a=h^2+3 y^2+C_2 \,,\\
& b^\pm=(3y\mp h)\left( \frac{(2-y\mp h)C_2}{4}-\frac{C_3}{3}+\frac{(y-2\pm h)(y+2\pm h)(y\mp h)}{8} \right)\,.
\end{align}
\end{subequations}
The Casimir elements $C_2$, $C_3$ of $U(sl_3)$ are defined by \eqref{CEsl3} and
$h=\tfrac{1}{2}(e_{22}-e_{33})$, $y=\tfrac{1}{6}(2e_{11}-e_{22}-e_{33})$ are Cartan
elements stable under $T$.

Relations \eqref{Racah-sl3} are those of the Racah algebra\footnote{
Up to a trivial rescaling of the generators such as $J=-2K_2$, $\overline{J}=-2K_1$,
$a=4d$, $b^{+}=8e_1$, $b^{-}=8e_2$.
}.
This realization interestingly adds to others previously found, namely between the
recurrence and the difference operators of the Racah polynomials \cite{Zhedanov1991,
GenestVinetetal2014, GenestVinetetal2014a},
in terms of the intermediate Casimir elements of $U(sl_2)^{\otimes 3}$
\cite{GranovskiiZhedanov1988} or in terms of the elements of $U(sl_2)$
\cite{GranovskiiZhedanov1993a, CrampeShaabankabakibo2021}.
The result here provides a realization of the Racah algebra in terms of the elements
of $U(sl_3)$ (also see \cite{CorreaDelolmoetal2021,
CampoamorstursbergLatinietal2023}).

It is well-known that there exists a central element in the Racah algebra given by
\begin{equation}
 \Gamma=2\{J^2,\overline{J}\}+2\{J,\overline{J}^2\}-K^2-4( J+\overline{J})^2
 -a\{J,\overline{J}\}+2(b^- +a)J+2(b^+ + a)\overline{J}\, .
\end{equation}
In the realization \eqref{Racah-sl3}-\eqref{eq:realization_racah_central_elements}
introduced here, this central element takes a definite value in terms of the
other central elements of the Racah algebra:
\begin{align}\label{sRacah-sl3}
 \Gamma=&\frac{1}{2}(y^2+h^2)^2(1-C_2) -\frac{1}{8}(h^2-y^2)^3+2 h^2 y^2  -\left(\frac{C_3}{3}+2y+y C_2\right)\left( \frac{C_3}{3} +2y+yC_2-C_2\right) \nonumber\\
 &-\frac{y}{6}(3h^2-11y^2)(3C_2-2C_3)+\frac{C_2}{8}(h^2+3y^2)(5h^2-y^2+4)\,.
\end{align}
The quotient of the Racah algebra by the previous relation \eqref{sRacah-sl3} is called
the special Racah algebra \cite{CrampeGaboriaudetal2021}.
The terminology ``special'' comes from \cite{CrampeFrappatetal2021}, where a similar
quotient, called ``special Askey–Wilson algebra'', was introduced for the
Askey--Wilson algebra.
This type of quotient appears in various other situations and, in particular, as the
diagonal centralizer of $U(sl_2)$ in $U(sl_2)^{\otimes 3}$ \cite{CrampeGaboriaudetal2021}.

\paragraph{Isomorphism with a centralizer.}
In this paragraph, the centralizer $\mathcal{Z}$ of the Cartan subalgebra in $U(sl_3)$ is
described using the elements previously introduced.
The centralizer is defined as follows:
\begin{align}
 \mathcal{Z}=\{x\in U(sl_3)\ |\ [x,H_{i}]=0\,,\ i=1,2\}\,,
\end{align}
where $H_1=\tfrac12(e_{11}-e_{22})$ and $H_2=\tfrac12(e_{22}-e_{33})$.
Assigning degree $1$ to each generator $e_{ij}$, the algebra $U(sl_3)$ is filtered and so
is its subalgebra $\mathcal{Z}$, which allows to describe $\mathcal{Z}$ as in the
following proposition.
\begin{prop}
The algebra $\mathcal{Z}$ is generated by
\begin{equation}\label{generatorsZ}
J\, ,\overline J\, , K\,, H_1,\, H_2,\, C_2,\, C_3\,,
\end{equation}
and its Hilbert--Poincar\'e series is:
\begin{equation}\label{HPZ}
F_{\mathcal{Z}}(t)=\frac{1+t^3}{(1-t)^2(1-t^2)^3(1-t^3)}\,.
\end{equation}
\end{prop}
\begin{proof}
A monomial $H_1^aH_2^b\,e_{21}^c\, e_{12}^d\, e_{32}^e\, e_{23}^f\, e_{31}^g\, e_{13}^h$
commutes with $H_1$ and $H_2$ if and only if:
\[\left\{\begin{array}{l}
2(d-c)+(e-f)+(h-g)=0\,,\\
(c-d)+2(f-e)+(h-g)=0\,,\end{array}\right.\qquad d-c=f-e=h-g\,.\]
Considering either $d-c\geq 0$ or $d-c\leq 0$, it follows that $\mathcal{Z}$ is spanned by
the following elements:
\begin{equation}\label{generatingset}
\{H_1^aH_2^b\bigl(e_{21}e_{12}\bigr)^i\bigl(e_{32}e_{23}\bigr)^j\bigl(e_{31}e_{13}\bigr)^kx\}_{a,b,i,j,k\geq 0}\,,\ \ \text{with}\ x\in\{\bigl(e_{12}e_{23}e_{31}\bigr)^{\ell},\bigl(e_{32}e_{21}e_{13}\bigr)^{\ell}\}_{\ell\geq0}\,.
\end{equation}
This shows that $\mathcal{Z}$ is generated as an algebra by:
\begin{equation}
 H_1\,,\ H_2\,,\ e_{21}e_{12}\,,\ e_{32}e_{23}\,,\ e_{31}e_{13}\,,\ e_{31}e_{12}e_{23}\,,\ e_{32}e_{21}e_{13}\,.
\end{equation}
By direct computations, one extracts
\begin{subequations}
\begin{align}
 &e_{21}e_{12}=J -H_1(H_1+1)\,,\\
 &e_{31}e_{13}= \overline J -(H_1+H_2)(H_1+H_2+1)\,,\\
 & e_{32}e_{23} =\tfrac{1}{2}C_2 -J- \overline{J}
  -H_2+\tfrac{1}{3}(2 H_1H_2+2H_1^2-H_2^2)  \,,\\
& e_{31}e_{12}e_{23}-e_{32}e_{21}e_{13}=K\,,\\
&  e_{31}e_{12}e_{23}  + e_{32}e_{21}e_{13}
 = \tfrac{1}{3}C_3-2J(H_1+H_2+1)-2\overline{J}H_1+\tfrac{1}{3}C_2(2H_1+ H_2)\nonumber\\
&\qquad +\tfrac{2}{27}(H_2+2H_1)(11 H_1^2+11H_1H_2-4H_2^2)
 -\tfrac{2}{3} (2H_1H_2- H_1^2+2H_2^2+2H_2-H_1)\,.
\end{align}
\end{subequations}
It follows easily that $\mathcal{Z}$ is generated by the elements in (\ref{generatorsZ}).

To calculate the Hilbert--Poincar\'e series, we sum the degrees of the elements in (\ref{generatingset}), which are linearly independent in $U(sl_3)$. Namely, one gets
\begin{align}
 F_{\mathcal{Z}}(t)=\sum_{a,b,i,j,k\geq 0}t^{a+b+2(i+j+k)}\left(1+2\sum_{\ell\geq 1}t^{3\ell}\right)
=\frac{1}{(1-t)^2(1-t^2)^3}\left(1+\frac{2t^3}{(1-t^3)}\right)\,,
\end{align}
which leads to the desired result.
\end{proof}

Recall that the elements $H_1,H_2,C_2,C_3$ are central in the subalgebra $\mathcal{Z}$. Moreover, we have found some generators of $\mathcal{Z}$ together with the relations (\ref{Racah-sl3}) and (\ref{sRacah-sl3}).
These relations allow to write any element of $\mathcal{Z}$ as a linear combination of
\begin{equation}\label{basis-Zc}
\{H_1^aH_2^bC_2^cC_3^d J^i \overline{J}^j K^k\}_{a,b,c,d,i,j\geq 0,\ k\in\{0,1\}}\,.
\end{equation}
This is true since we can reorder any product using (\ref{Racah-sl3}), the elements
$H_1,H_2,C_2,C_3$ are central, and $K^2$ is rewritten using the special relation
(\ref{sRacah-sl3}).

Thus, the above set is a spanning set of $\mathcal{Z}$, and comparing with the
Hilbert--Poincar\'e series found previously, we arrive at the following corollary.
\begin{coro}
The centralizer $\mathcal{Z}$ is isomorphic to the algebra generated by $J$,
$\overline{J}$, $K$ and the central elements $H_1,H_2,C_2,C_3$ with the defining relations
(\ref{Racah-sl3}) and (\ref{sRacah-sl3}). A basis is given by the set (\ref{basis-Zc}).
\end{coro}

\section{Conclusion}\label{sec:conclusion}
The realization of the rank one Racah algebra as the centralizer in $U(\sl_3)$ of the Cartan
subalgebra naturally brings to mind its generalization to $U(\sl_n)$ for any $n$.
This suggests a way of realizing higher rank Racah algebras which is different from the
generalization related to tensor products of $U(\sl_2)$.
Centralizers of the Cartan subalgebras were studied in
\cite{CampoamorstursbergLatinietal2023} but a description of the resulting algebra is
still not known for arbitrary $n$.
Similar results for any simple Lie (super)-algebra should be also very interesting.
We have so far focused on finite-dimensional representations of the algebra and
obtained correspondingly finite families of functions.
Infinite-dimensional representations should provide similar results, involving
infinite families of functions.

Studying the cases associated to the quantum groups can also be envisaged. The results
of this paper should generalize: the centralizer of the Cartan subalgebra of $U_q(sl_3)$
should be associated to the Askey--Wilson algebra, the Krawtchouk polynomials should be
replaced by certain $q$-Krawtchouk polynomials (see \cite{GenestPostetal2016,
GenestPostetal2017}) and the Racah polynomials, by the $q$-Racah polynomials.
We plan on pursuing these questions.

\subsection*{Acknowledgments}

N.~Crampé and L.~Poulain d'Andecy are partially supported by Agence Nationale de la
Recherche Projet AHA ANR-18-CE40-0001 and by the IRP AAPT of CNRS.
J.~Gaboriaud held an Alexander-Graham-Bell scholarship from the
Natural Sciences and Engineering Research Council of Canada (NSERC),
received scholarships from the ISM and the Universit\'e de Montr\'eal
and is now supported by JSPS KAKENHI Grant Numbers 22F21320 and 22KF0189.
The research of L.~Vinet is funded in part by a Discovery Grant from NSERC.
\\
% Any opinions, findings, and conclusions or recommendations expressed in this material
% are those of the author(s) and do not necessarily reflect the views of the author(s)’
% organization, JSPS or MEXT.

\subsection*{Data Availability}
Data sharing not applicable – no new data generated.

\printbibliography

\end{document}